\documentclass[12pt]{amsart} 
\allowdisplaybreaks
\usepackage{float}
\usepackage{color}
\usepackage{amssymb}
\usepackage{times}
\usepackage{amsmath}
\usepackage[pdftex]{graphicx}
\usepackage{hyperref}
\usepackage[square,numbers]{natbib}
\usepackage{mathrsfs}
\usepackage{natbib}
\usepackage[margin=1.25in]{geometry}
\usepackage{soul, comment} 
\newcommand{\comm}[1]{\ignorespaces}

\newtheorem{thm}{Theorem}[section]
\newtheorem*{theorem_12}{Theorem 1.2}
\newtheorem{lem}[thm]{Lemma}
\newtheorem{prop}[thm]{Proposition}
\newtheorem{cor}[thm]{Corollary}   
\newtheorem{definition}[thm]{Definition}

\newtheorem{example}[thm]{Example}

\renewcommand{\int}{\mathrm{int}}
\newcommand{\R}{\mathbb{R}}
\newcommand{\Z}{\mathbb{Z}}

\newcommand{\vol}{\operatorname{vol}}

\newcommand{\diam}{\operatorname{diam}}
\newcommand{\Ric}{\operatorname{Ric}}

\renewcommand{\rho}{\varrho}

\definecolor{purple}{rgb}{0.56, 0, 1}
\definecolor{teal}{rgb}{0, 0.55, 0.55}

\newtheorem*{CSL}{Curve Shortening Lemma}

\title[Orthogonal Geodesic Chords]{Lengths of Orthogonal Geodesic Chords on Riemannian Manifolds}
\date{\today}

\begin{document}

\author[Beach]{Isabel Beach}
\address[Beach]{Department of Mathematics,  University of Toronto, Toronto, Canada}
\email{isabel.beach@mail.utoronto.ca}

\author[Contreras-Peruyero]{Hayde\'e Contreras-Peruyero}
\address[Contreras-Peruyero]{Centro de Ciencias Matemáticas, UNAM, Morelia, Mexico}
\email{haydeeperuyero@matmor.unam.mx}

\author[Griffin]{Erin Griffin}
\address[Griffin]{Department of Mathematics, Northwestern University, Evanston, Illinois}
\email{griffine@northwestern.edu}

\author[Rotman]{Regina Rotman}
\address[Rotman]{Department of Mathematics,  University of Toronto, Toronto, Canada}
\email{rina@math.toronto.edu}

\author[Searle]{Catherine Searle}
\address[Searle]{Department of Mathematics, Statistics, and Physics, Wichita State University, Wichita, Kansas}
\email{searle@math.wichita.edu}

\begin{abstract} 

Let $N$ be a closed submanifold of a complete manifold, $M$. Then under certain topological conditions, there exists an orthogonal geodesic chord beginning and ending in $N$. In this paper we establish an upper bound for the length of such a geodesic chord in terms of geometric bounds on $M$. For example, if $N$ is a $2$-dimensional sphere embedded in a closed Riemannian $n$-manifold, then there exists an orthogonal geodesic chord in $M$ with endpoints on $N$ that has length at most 
$$
(4d+96D +8232\sqrt{A})(2n+1)
$$ 
where $d$ is the diameter of $M$, and $A$ and $D$ are the area and intrinsic diameter of $N$, respectively.
\end{abstract}

\maketitle

\section{Introduction} 
Given a manifold $M$ and a closed submanifold $N$, a solution to 
the {\em minimal submanifold free boundary problem} for $N \subset M$ is a minimal submanifold $W\subset M$ such that $\partial W\subset N$ and $ W$ meets $N$ orthogonally at $\partial W$. When the ambient manifold $M$ has boundary, $N$ is typically required to be $\partial M$. If $N=\partial M$, we additionally require $\int(W)\cap N= \emptyset$. In this paper, we focus on the case where $W$ has dimension one, in which case we say that $W$ is an {\em orthogonal geodesic chord}. We are interested in establishing an upper bound for the length of an orthogonal geodesic chord in terms of the geometric parameters of $M$ and $N$.
\par
Throughout, we assume that $M$ is a complete $n$-dimensional manifold, and $N \subsetneq M$ is a closed submanifold that is not a deformation retract of $M$. 
Let $\Omega_NM$ be the space of  piecewise-differentiable curves in $M$ with endpoints on $N$, and view $N$ as a subset of $\Omega_NM$ by identifying it with the set of point curves on $N$. 
Consider a representative $f$ of a non-trivial homotopy class $[f] \in \pi_i(\Omega_NM,N)$. 
This map can be viewed as a map $f:(D^i\times I, \partial (D^i\times I)) \to (M, N)$ by parameterizing it by $t\in I$. 
In order to extend our result to complete, non-compact manifolds, we require a bound on the size of the image of $f$. Our result therefore relies on a choice of a subset $V\subset M$ with finite extrinsic diameter $d$ and that contains both $N$ and the image of $f$. Our result also depends on a natural geometric property of $V$ that roughly quantifies how easy it is to significantly shorten a long loop in $V$ based at a point in $N$, which we call the {\em effective loop shortening property}, see Definition \ref{def:shortening}.
With this in mind, 
our main result is as follows.

\begin{thm} 
    \label{MainTheorem1}
   Let $M$ be a complete $n$-dimensional manifold, and $N \subsetneq M$ a closed submanifold such that $\pi_i(\Omega_NM,N)\neq 0$ for some $i>0$. 
   Let $V$ be a subset of $M$ that has bounded extrinsic diameter $d$ and contains both $N$ and the image of $f$. Suppose $V$ satisfies the effective loop shortening property restricted to $N$ with parameters $c, a,$ and $k$, where, by definition, $c>2kd$.
   Then there exists a geodesic chord that is orthogonal to $N$ of length at most $2c(2i+1)$. 
\end{thm}
    \noindent 
    Note that this result applies to any such pair $f$ and $V$, but of course it is ideal to choose $f$ and $V$ so that the extrinsic diameter of $V$ realizes the infimum
\begin{align*}
d_{\inf} &
=\inf_{F\in [f] }\inf_{V\subseteq M}
\{\diam_M(V)\mid F(D^i \times I)\cup N\subseteq V\},
\end{align*}
where $\diam_M(V)$ is the extrinsic diameter of $V$ in $M$.
This is the same as taking $V$ of the form $f(D^i \times I)\cup N$ and choosing $f$ so that the diameter of $f(D^i \times I)\cup N$ realizes the infimum
\begin{align*}
\inf_{F\in [f] } \{\diam_M(F(D^i \times I)\cup N)\}.
\end{align*}
Note that if $M$ is closed, then $d$ and $d_{\inf}$ are both bounded above by the diameter of $M$.
\par 
There are many settings where the parameters of the effective loop shortening can be estimated, leading to a number of applications of Theorem \ref{MainTheorem1}. We discuss most of these applications in Section \ref{sec:applications}. One especially interesting consequence of the above theorem is the following, which we prove in Section \ref{sec:applications}.
\begin{thm} 
    \label{thm:2sphere_in_ndisk} 
    Let $M$ be a closed Riemannian manifold of dimension $n$ and diameter $d$. Let $N$ be a $2$-dimensional sphere of area $A(N)$ and intrinsic diameter $D$ embedded in $M$. Then there exists a geodesic chord on $M$ orthogonal to $N$ of length at most 
    $$
    (4d+96D +8232\sqrt{A(N)})(2n+1).
    $$
\end{thm}

\subsection{Background and Motivation}

The history of the free boundary problem is detailed by S. Hildebrandt in \cite{hildebrandt1986}. In recent years, there have been a great number of results establishing the existence of a free boundary minimal surface or hypersurface under various topological and geometric conditions. For example, A. Fraser showed in \cite{fraser2000} that given a closed Riemannian manifold $M$ and a closed submanifold $N$, there exists a free boundary Riemannian $2$-disk with boundary on $N$ with bounded Morse index, provided that $M$ and $N$ satisfy some natural topological conditions. In \cite{lin2020}, L. Lin, A. Sun and X. Zhao generalize this result to all closed ambient Riemannian manifolds  using methods from T. Colding and W. Minicozzi, \cite{coldingminicozzi2008}. J. Chen, A. Fraser, and C. Pang generalize this result to free boundary surfaces of higher genus in \cite{fraser_chen_pang_2015}. In \cite{manchunli2015}, M. Li uses methods from T. Colding and C. De Lellis, \cite{colding_delellis_2003}, to prove that any compact $3$-dimensional Riemannian manifold $M$ with boundary admits a properly embedded free boundary minimal surface. The existence of a free boundary minimal annulus in a convex, compact 3-dimensional manifold with boundary was proven by D. Maximo, I. Nunes, and G. Smith in \cite{maximo2017}. 
\par
There are also a number of results that deal with the existence of free boundary minimal surfaces in a ball. The groundbreaking results of A. Fraser and R. Schoen (see \cite{frasch2011, frasersch2016}) demonstrate that the solution of a certain extremal Steklov eigenvalue problem on a compact surface with boundary can be used to generate a free boundary minimal surface with genus zero and $n$ boundary components in the unit ball. A similar result was proven by A. Folha, F. Pacard, and T. Zolotareva in \cite{folpaczol2017} for a three-dimensional ball and free boundary surfaces of genus one with a large number of boundary components. In the annulus, the work in \cite{frasch2011} was generalized by X.-Q. Fan, L.-F. Tam, and C. Yu in \cite{fantamyu2015}. For more results and open questions in this area, we refer the reader to a recent survey of M. Li \cite{li2020_openquestions}.
\par
Furthermore, one can study the existence of min-max minimal surfaces, which are surfaces whose area equals the min-max value over all sweepouts of a manifold by disks whose boundary lies in a given submanifold. For example of a result in this area, see the work of P. Laurain and R. Petrides \cite{laurain2019}. 
\par
The existence of free boundary minimal hypersurfaces is also well-studied: for example, in C. De Lellis and J. Ramic \cite{delellisramic2018}, M. Li and Zhou \cite{li2017}, and Z. Wang \cite{wang2019}. In \cite{li2017}, M. Li and Zhou proved the existence of infinitely many free boundary minimal hypersurfaces in a manifold with non-negative Ricci curvature. In \cite{wang2021}, Z. Wang solved a free boundary version of Yau's conjecture, establishing the existence of infinitely many free boundary minimal hypersurfaces in any compact Riemannian manifold with smooth boundary and of dimension at least three and at most seven. However, these hypersurfaces are only almost properly embedded. This result follows from a series of papers due to Y. Liokumovich \cite{liokumovich2018}, F. Marques and A. Neves \cite{marques2017}, K. Irie \cite{irie2018},  and A. Song  \cite{song_2023}, that solved S.-T. Yau's original conjecture for closed hypersurfaces (see \cite{yau}).
\par
Notably, orthogonal geodesic chords are one-dimensional free boundary minimal submanifolds in a manifold with boundary. In \cite{bos1963}, W. Bos proved the existence of $n$ orthogonal geodesic chords on an $n$-dimensional disk with convex boundary. In \cite{gluck1984}, H. Gluck and  W. Ziller proved the existence of an orthogonal geodesic chord on every $n$-dimensional manifold with convex boundary. In \cite{zhou2016}, X. Zhou uses the min-max method to prove the existence of an orthogonal geodesic chord in $M$ with endpoints on a closed submanifold $N$, focusing on the case where $\pi_1(M, N)=0$. In \cite{giambo2020}, R. Giamb{\`o}, F. Giannoni, and P. Piccione proved that there are at least $n$ orthogonal geodesic chords on an $n$-disk with strongly concave boundary.
J. Hass and P. Scott, in \cite{hass1994}, established the existence of two simple orthogonal geodesic chords on the convex 2-disk. Recently, D. Ko  gave another proof of this result in  \cite{ko2023}. Extending this result to higher dimensions, H. B. Rademacher proved that a generic metric on an $n$-disk with convex boundary admits $n$ simple orthogonal geodesic chords for $n\geq 3$ in \cite{rademacher2024}.
\par
Whenever the existence of free boundary minimal surfaces or hypersurfaces has been established, it is natural to study their topological and geometric properties, such as their Morse index, area, and curvature (see \cite{fraser2000, brendle2012, fraser2014, ketover2016,kapouleas2017_critcat, sargent2017, smith2017, guang2018, guang2019, carlotto2020, tran2020, shulichen2021, guang23, zhu2023,matinpour2024}). The first author proved length bounds for the length of the shortest two simple orthogonal geodesic chords in a convex 2-disk \cite{beach2024_2d_case}, building on their work of \cite{beach2024} on length bounds for the length of the shortest two simple geodesic loops in 2-sphere.
\par
Apart from their role as solutions to the minimal submanifold free boundary problem, there is interest in the study of orthogonal geodesic chords due to their link with Hamiltonian systems. Consider coordinates $(p, q) \in \R^n \times \R^n$ and a Hamiltonian
$$
H(p,q) = \frac{1}{2} p^T A(q) p + V(q),
$$
for some non-degenerate 2-form $A: \R^n \to  \R^n \times \R^n$ and a function $V: \R^n \to  \R^n$. Solutions to the original Hamiltonian system correspond to solutions of the equation
\begin{equation}
    \label{eq:hamiltonian}
0 = \frac{D^2}{ dt^2} q + \nabla V(q)
\end{equation}
on $(\R^n,g)$ with metric $g(q) = \frac{1}{2} A(q)$. Additionally, solutions of Equation \ref{eq:hamiltonian} have constant ``energy" given by $E= \frac{1}{2} ||q||^2 + V(q)$. A certain class of solutions with pendulum-like periodic motion are known as {\em brake orbits}.
The Maupertuis principle states that brake orbits with energy $E$ are geodesics with respect to the Jacobi metric $g_E=(E-V)g$ that lie in the set $V^{-1}((-\infty, E])$ (see e.g. \cite{giambo2004} by R. Giamb{\`o}, F. Giannoni, and P. Piccione). Unfortunately, $g_E$ vanishes on the boundary $V^{-1}(E)$ and is consequently not Riemannian. However, it is shown in \cite{giambo2004} that if $E$ is a regular value of $V$ and $V^{-1} ((-\infty, E])$ is compact, then for sufficiently small $\delta>0$ an orthogonal geodesic chord on the Riemannian manifold $V^{-1} ((-\infty, E- \delta])$ corresponds to a unique brake orbit on $V^{-1}((-\infty, E])$.
\par
In \cite{seifert_1949}, H. Seifert proved the existence of at least one brake orbit when $V^{-1} ((-\infty, E])$ is homeomorphic to an $n$-dimensional disk. Furthermore, he conjectured that there are at least $n$ geometrically distinct brake orbits of energy $E$ in this case. This conjecture has been investigated by R. Giamb{\`o}, F. Giannoni, and P. Piccione in a series of papers \cite{giambo2010, giambo2011, giambo2015, giambo2018, giambo2020}, and C. Liu, Y. Long, D. Zhang, and C. Zhu in \cite{liu_2013, liu_2014, long_2006, zhang_2011}. The conjecture  was ultimately proven in \cite{giambo2020}.

\subsection{Outline of Proof of Theorem \ref{MainTheorem1}}

We first establish a topological criteria for the existence of a sweepout of $M$ by curves with endpoints on $N$. We show that, assuming that $N$ is not a deformation retract of $M$, there exists $i>0$ such that $\pi_i(\Omega_NM,N)\neq 0$. We then apply the {\em min-max principle} as follows. We consider a sequence of representatives of a non-trivial homotopy class in $\pi_i(\Omega_NM,N)$. In the limit, the maximum length of a curve in these representatives tends to the smallest possible value for the maximum length of a curve in a representative of this class.
The longest such curve in each representative converges to a geodesic chord orthogonal to $N$.
If we are able to construct a particular representative that only passes through curves of length bounded above by a uniform constant, then we have an upper bound on the length of a shortest orthogonal geodesic chord. 
\par 
We want our bound to depend only on geometric parameters of $M$ and $N$. 
In particular,  our bound is defined in terms of the effective loop shortening property. This property quantifies how long a loop with a basepoint in $N$ must be to ensure that we can always ``significantly" shorten it through loops of controlled length. 
We show in the \hyperref[CSL]{Curve Shortening Lemma} that, assuming there are no sufficiently short stable geodesic chords in $M$ orthogonal to $N$, the effective loop shortening property can be reduced to an analogous geometric property of $N$. This ability to shorten loops allows us to construct the desired short sweepout by applying techniques from the work of the first, second, fourth and fifth authors in \cite{linear_bounds_group_paper}.
\par
We note, however, that there are examples demonstrating that it is not always possible to control the lengths of curves in such sweepouts. For instance, based on the example of S. Frankel and M. Katz in \cite{frankelkatz1993}, Y. Liokumovich demonstrated the existence of a sequence of metrics on $S^2$ with
bounded diameter whose ``optimal'' sweepouts contain curves of arbitrarily large length \cite{liokumovich2012}. 
This result was extended to $S^3$ by O. Alshawa and H. Y. Cheng in \cite{alshawacheng2024}, where they proved that there exists a sequence of metrics on $S^3$ such that, for some submanifold $N$ diffeomorphic to $S^2$, any sweepout of a non-trivial homotopy class in $\pi_1(\Omega_NS^3, N)$ must contain a long curve. 
\par
Consequently, it is not always possible for us to construct a sweepout that passes only through short curves. This corresponds to the case where we can only apply the effective loop shortening property when at least one of the parameters $c, a,$ or $k$ is very large. Intuitively, this means that it is not possible to efficiently shorten loops of length $2ak$. On the other hand, the obstruction to efficiently shortening loops is the existence of an orthogonal geodesic chord of length at most $2ak$: if no such chord exists, we could apply a loop shortening flow to shorten any loop of length $2ak$ through curves of length at most $2ak$ by the \hyperref[CSL]{Curve Shortening Lemma}. Therefore,  we cannot construct a sweepout that passes through short curves if $M$ already admits a short orthogonal geodesic chord, in which case our result holds anyway.
\par
Finally, note that when $N$ is a point $p$, we can interpret a geodesic chord orthogonal to $N$ as a geodesic loop based at $p$. Thus, our bounds for the length of a shortest orthogonal geodesic chord can be viewed as a generalization of the length bounds for geodesic loops obtained, for example, by S. Saborau in \cite{sabourau2004_fillrad}, by the fourth author in \cite{rotman_2008_loops}, by A. Nabutovsky and the fourth author in \cite{rotman_quant}, and by the first author in \cite{beach2024}.

\subsection{Organization}

This paper is organized as follows. In Section \ref{2}, we establish topological criteria for the existence of a non-trivial homotopy class in $\pi_i(\Omega_NM, N)$, which guarantees the existence of a non-trivial sweepout. Section \ref{sec:main_results} is dedicated to the proof of Theorem \ref{MainTheorem1}. First, we describe the effective loop shortening property and how it can be used to connect points on a loop to the base point by short curves. In the absence of short orthogonal geodesic chords, we show how the effective loop shortening property on a submanifold can be upgraded to the ambient space. We then prove how to modify a representative of a non-trivial homotopy class in $\pi_i(\Omega_NM, N)$ to obtain a homotopic map that passes only through short curves. Theorem \ref{MainTheorem1} is then proven  by an application of standard min-max techniques. Finally, in Section \ref{sec:applications}, we establish a number of interesting applications of Theorem \ref{MainTheorem1}.

\section{Existence of a Non-Trivial Relative Homotopy Class}
\label{2}

In this section we establish topological criteria for the existence of a sweepout. Recall that a non-contractible map $f:(D^i, \partial D^i)\rightarrow (\Omega_NM, N)$ is called a {\em sweepout of a non-trivial class of $\pi_i(\Omega_nM, N)$}.
\par 
 Recall that any map $$(D^i, \partial D^i)\rightarrow (M, N)$$ naturally induces a map $$(D^{i-1}, \partial D^{i-1})\rightarrow (\Omega_N M, N)$$ and vice versa. Thus, we have $\pi_i(M, N)\cong \pi_{i-1}(\Omega_N M, N)$, see, for example, Section 4.J of A. Hatcher's textbook, 
 \cite{hatcher}. If there exists a smallest number $i$ such that $\pi_i(M,N) \neq \{0\}$ and $1<i<\infty$,  then $\pi_{i-1}(\Omega_NM,N) \neq \{0\}$. 
 
\par 
In the case that $\pi_i(M,N) = \{0\}$ for all $i$, an orthogonal geodesic chord may not exist. For example, take $M$ to be the standard cylinder $S^1 \times \mathbb{R}$ and take $N=S^1 \times \{0\}$. In this case, the fundamental groups of $N$ and $M$ are isomorphic and $M$ and $N$ are both aspherical. Therefore the exact sequence of the pair  shows that $\pi_i(M,N) = \{0\}$ for all $i$. On the other hand, there are no geodesic chords orthogonal to $N$ because the only geodesics in $M$ orthogonal to $N$ are the lines $t\mapsto (\theta,t)$ for any fixed $\theta\in S^1$. 
\par 
We now provide some cases for which $\pi_i(M,N) \neq\{0\}$ for some $1\leq i\leq n$, where $n$ is the dimension of $M$.
We first consider the case where $M$ is complete and $N$ is a closed submanifold of $M$.  The Compression Lemma in Chapter 4 of \cite{hatcher} 
 states that for any CW pair $(X, A)$ and  $(Y , B)$ any pair with $B\neq \emptyset$, if we assume that $\pi_n(Y, B, y_0)=0$ for
all $y_0\in B$ for
each $n$ such that $X \setminus  A$ has cells of dimension $n$,   then every map $f : (X, A)\rightarrow (Y, B)$ is homotopic rel A to a map $X\rightarrow B$.
Since $(M, N)$ is a CW pair, if we assume that the inclusion of $N$ in $M$ is a homotopy equivalence, then the identity map from $(M, N)$ to itself satisfies the hypotheses of the Compression Lemma and  so $N$ is a deformation retract of $M$. 
Conversely, if $N$ is a deformation retract of $M$, then clearly the inclusion of $N$ in $M$ is a homotopy equivalence, and we obtain the following corollary.

\begin{cor}
    \label{cor:defretract}
    Let $M$ be a complete manifold and $N\subsetneq M$ a closed submanifold. Then the inclusion of $N$ in $M$ is a homotopy equivalence if and only if $N$ is a deformation retract of $M$.
    In particular,  $\pi_i(M, N)$  is non-trivial for some $i\geq 1$ if and only if $N$ is not a deformation retract. \qed
\end{cor}

In the particular case where $M$ is closed, we have the following result.
\begin{prop} 
\label{prop:pi_non_zero}
    Let $M$ be a closed  $n$-manifold and let $N\subsetneq M$ be a closed submanifold. Then $N$ is not a deformation retract of $M$. 
\end{prop}
\begin{proof}
    We argue by contradiction. Suppose that $N$ is a deformation retract of $M$. Then the inclusion of $N$ in $M$ is a homotopy equivalence and hence $H_i(M)\cong H_i(N)$ for all $i\leq \dim(N)$.  
    We use $\Z_2$-coefficients, as we make no assumptions on orientability. It suffices to show that $H_i(M; \Z_2)$ is not isomorphic to $H_i(N; \Z_2)$ for some $i>0$. Since $N$ is a closed submanifold of $M$, $\dim(N)<\dim(M)$ and hence $H_n(N;\Z_2)=0$, whereas $H_n(M; \Z_2)\cong \Z_2$, a contradiction, and the result follows.
\end{proof}

Finally, we consider the case where we have a compact manifold, $M$, with a unique boundary component, $N$ in the following proposition. 

\begin{prop}
\label{boundary}
    Let $M$ be a compact manifold with connected boundary. Then $M$ does not deformation retract onto $\partial M$. \qed
\end{prop}
\begin{proof}
    Consider the double of $M$, $DM=M\cup_{\partial M} M$. Suppose $M$ retracts onto $\partial M$ and let $r:M\times I\rightarrow M$ be the retraction. Then by the Pasting Theorem 18.3 in Munkres \cite{Mu}, it follows that $r\smile r$ gives us a retraction of $DM$ onto $\partial M$. However, this is impossible by Proposition \ref{prop:pi_non_zero}, and the result follows.
 \end{proof}

\section{Proofs of Main Results}
\label{sec:main_results}

Let $\Omega_N^LM$ denote the subset of $\Omega_NM$ of curves with length at most $L$,  where we identify $N$ with the space of point curves on $N$.
Let $f$ be a {\em sweepout of a non-trivial class of $\pi_i(\Omega_nM, N)$}, that is,   a non-contractible map $f:(D, \partial D)\rightarrow (\Omega_NM, N)$.  
If $M$ satisfies the effective loop shortening property with parameters $c$, $a$, and $k$ we show that there exists a 
map $g:(D^{i-1},\partial D^{i-1}) \to (\Omega_N^LM,N)$ that is homotopic to $f$ in the space $(\Omega_NM,N)$, where $L$ is a function of $c$ and $i$.
The map $g$ is constructed by splitting $D^{i-1}$ into very small rectangles and defining $g$ by induction on the skeleta of $D^{i-1}$. At each step we replace the curves in the image of $f$ on the cells with ``short'' curves. We  then show that $f$ and $g$ are homotopic. 
\par 
The value of $L$ depends on the property of being able to significantly shorten loops of sufficiently short length without increasing  their length by a large amount.
We call this property the {\em effective loop shortening property}. 
\begin{definition}[{\bf Effective Loop Shortening Property}]
    \label{def:shortening} 
    Let $V \subset M$ be a compact subset of a complete $n$-dimensional manifold $M$. Let $V$ have extrinsic diameter $d$.
    We say that $V$ satisfies the {\em effective loop shortening property with parameters $c, a$ and $k$} if there exist real numbers $c$, $a \geq d$ and $k\geq 1$ such that for sufficiently small $\delta>0$, every loop based at a point in $V$ of length at most $2ak+\delta$ can be shortened in $M$ to a loop of length at most $(2k-2)a$ over loops with a fixed basepoint of length at most $c$. Note that necessarily $c \geq 2ak+\delta$. 
    \par
    Let $K$ be a subset of $V$. If the above property is satisfied by all loops in $V$ with base points in $K$, we say that $V$ satisfies the {\em effective loop shortening property restricted to $K$ with parameters $c, a$ and $k$}.
\end{definition}
\noindent
We note that the effective loop shortening property is satisfied in many situations. We give one example here.
\begin{example} 
    Let $M$ be a closed Riemannian manifold with diameter $d$ such that there are no geodesic loops on $M$ of length less than or equal to $2d$. Then we claim that $M$ satisfies the effective loop shortening property with parameters $c=2d +\delta$, $a=d$, and $k=1$,  for an arbitrarily small $\delta$. To prove the claim, note that because $M$ is closed, it has a shortest geodesic loop and this loop has length strictly greater than $2d$. Therefore there is a small $\delta >0$ such that there exist no geodesic loops on $M$ of length in the interval $[0,2d+\delta]$. Then any closed curve of length at most $2d+\delta$ can be homotoped to a point via loops of length at most $2d+\delta$ by any curve shortening flow with a fixed basepoint, proving the claim. 
\end{example}
\par 
We now turn our attention to proving the main result.
We begin by proving the following lemma, which allows us to homotope between two arcs with equal endpoints if we can contract the closed loop they form.
\begin{lem}
    \label{lemma:loops_to_arcs} 
    Let $M$ be a Riemannian manifold. Let $e_1$ and $e_2$ be two paths connecting a pair of points $p, q \in M$ of respective lengths $l_1$ and $ l_2$. Suppose the loop $\alpha= e_1 * \overline{e}_2$  based at $p$ can be homotoped to a loop $\beta$ based at $p$ over loops $\alpha_\tau$ based at $p$ of length at most $l_3\geq l_1+l_2$, as in Figure \ref{fig:loops_to_arcs} (a). Then there exists a path homotopy between $\overline{\beta}*e_1$ and $e_2$ over curves of length at most $l_1 +l_3$. 
\end{lem}
\begin{figure}[ht]
    \includegraphics[width=\textwidth]{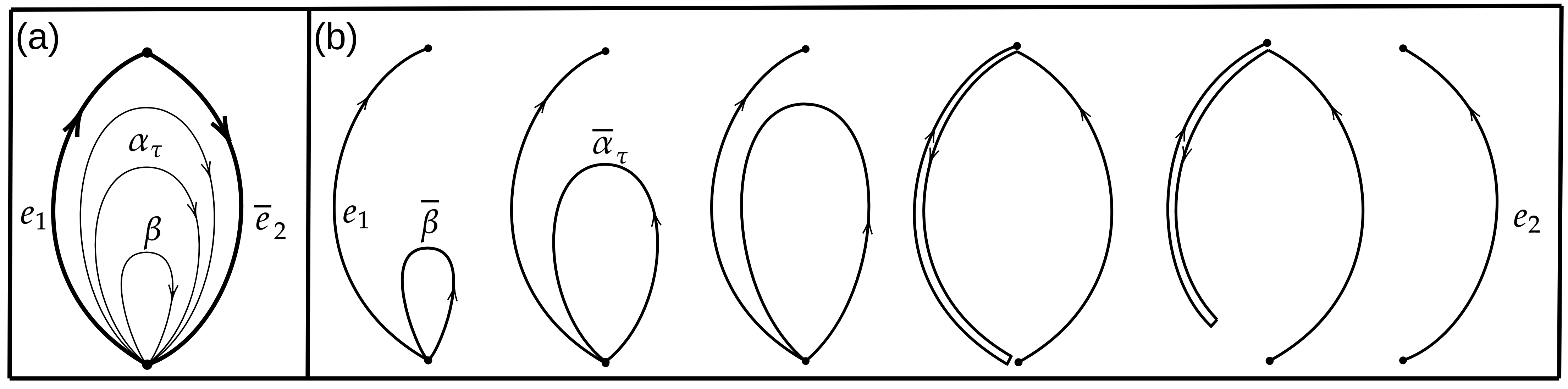}
    \caption{Illustration of the homotopies used in the proof of Lemma \ref{lemma:loops_to_arcs}.
\textbf{(a)} The path homotopy between $\alpha = e_1 \ast \bar{e}_2$ and $\beta$ through the family of loops $\alpha_\tau$.
 \textbf{(b)}\, The path homotopy between $\overline{\beta}*e_1$ and $e_2$. } 
    \label{fig:loops_to_arcs}
\end{figure}
\begin{proof}
  The homotopy between $\overline{\beta}*e_1$ and $e_2$ is as follows.
  The curves $\overline{\alpha}_\tau * e_1$ give a homotopy from $\overline{\beta}*e_1$ to $ e_2 * \overline{e}_1 * e_1$, see Figure \ref{fig:loops_to_arcs} (b). The curve $ e_2 * \overline{e}_1 * e_1$ is homotopic to $e_2$ because we can contract $ \overline{e}_1 * e_1$ along itself. This completes the desired homotopy. Moreover, the length of the longest curve in this homotopy is at most $l_1 + l_3$.
\end{proof}
\noindent
In a subset  $V\subset M$ that satisfies the effective loop shortening property with parameters $c, a$ and $k$, the above result allows us to replace a curve of length $2ka+\delta$ by a concatenation of a minimizing geodesic and a loop of length at most $ (2k-2)a$. This new curve is therefore shorter than the original by at least $\delta$. This fact is used in the following lemma, Lemma \ref{lemma:base_case}, to substantially shorten long curves, as in Lemma 5.2 of \cite{linear_bounds_group_paper}.
\par 
We remark that Lemma \ref{lemma:base_case} is a key ingredient which we use in the \hyperref[CSL]{Curve Shortening Lemma} to ``upgrade" the effective loop shortening property on a submanifold to the effective loop shortening property in the ambient space. Moreover, Lemma \ref{lemma:base_case} is used to prove Lemmas \ref{lemma:example_case} and  \ref{lemma:extending_delta_thin}. It also helps to establish the base case of the induction step of our main theorem in Lemma \ref{lemma:homotopic_to_short_map}.

 \begin{lem}
\label{lemma:base_case} 
     Let $M$ be a complete Riemannian manifold and $N$ a closed submanifold of $M$. Let $V$ be a subset of $M$ and $\eta$ be a path in $V$ beginning at $p\in M$. Suppose that one of the following holds.
     \begin{enumerate}
     \item
  $V$ satisfies the effective loop shortening property with parameters $c$, $a$, and $k$; or
    \item The path $\eta$ begins in $N$ and 
    $V$ satisfies the effective loop shortening property with parameters $c,a,$ and $k$ restricted to $N \subset V$. 
    \end{enumerate}
     Then there exists a one-parameter family of curves $\{\beta_t\}_{t\in [0, 1]}$ in $M$ such that $\beta_t(0)=p$ and $\beta_t(1)=\eta(t)$. Moreover, the length of each $\beta_t$ is at most $2c$ and if $\eta$ is a loop, then the length of $\beta_1$ is at most $(2k-2)a$.
 \end{lem} 

\begin{figure}[ht]
    \includegraphics[width=\textwidth]{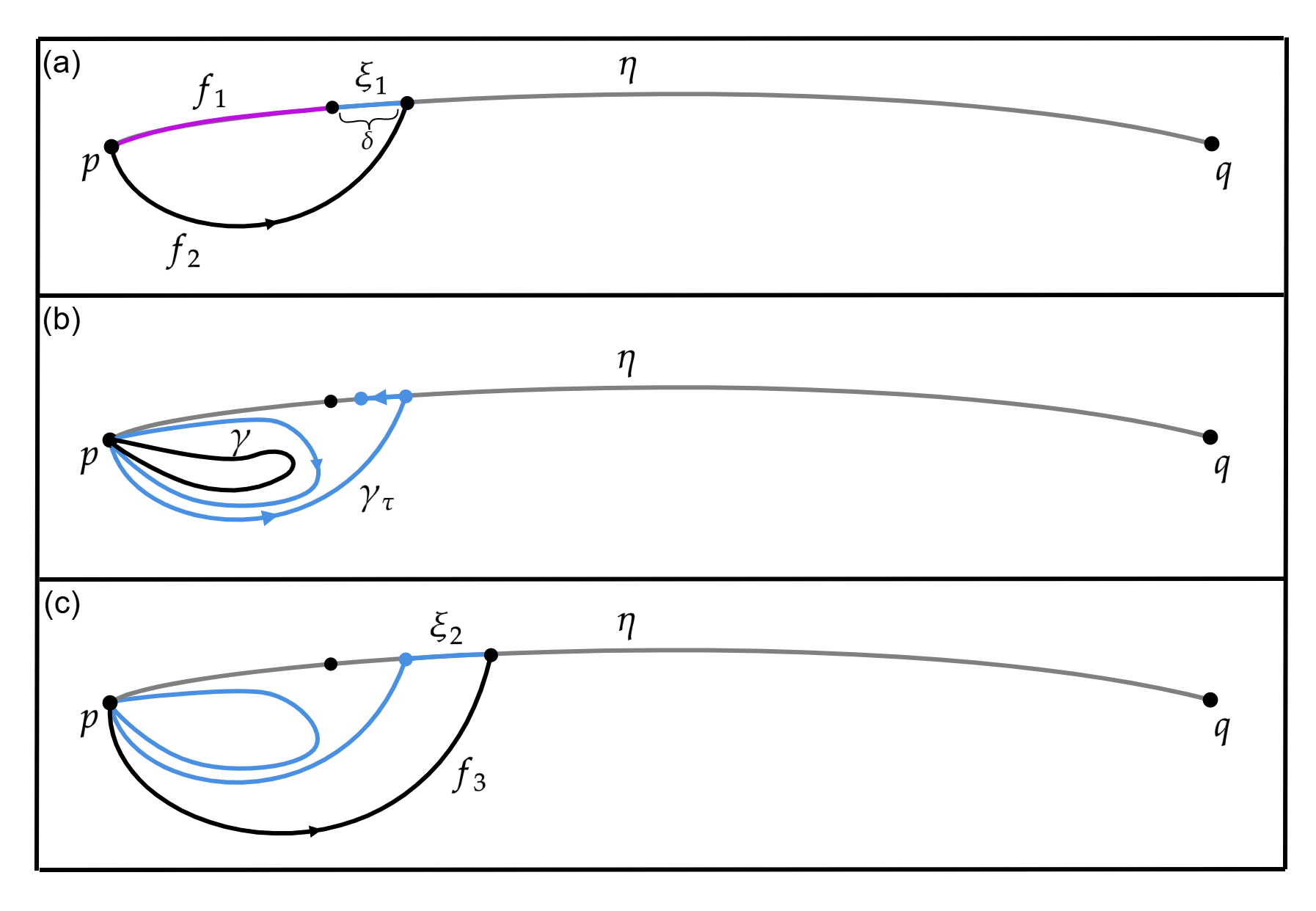}
    \caption{Illustration of the curves $f_i$ and $\xi_i$ in the proof of Lemma \ref{lemma:base_case}. } 
    \label{Fig2}
\end{figure}

\begin{proof}
    The proof of Part 2 follows immediately once we have proven Part 1 by making the additional assumption that the path $\eta$ starts at a point in $N$. We therefore proceed to prove Part 1. If $\eta$ is of length at most 
    $(2k-1)a$, we can simply take $\beta_t=\eta|_{[0,t]}$. Otherwise, we do the following. Let $d$ denote the extrinsic diameter of $V$ as measured in $M$. Choose $\delta$ small enough to satisfy Definition \ref{def:shortening} and such that $(2k-1)a+\delta\leq \ell(\eta)$.
    Parameterize $\eta$ by its arclength and define $f_1$ as the arc of $\eta$ from $\eta(0)$ to $\eta((2k-1)a)$. 
    Let $\xi_1$ be the arc of $\eta$, still parameterized by arclength, from $\eta((2k-1)a)$ to $\eta((2k-1)a+\delta)$. Connect $\eta((2k-1)a+\delta)$ and $\eta(0)$ by a minimizing geodesic in $M$, which we  denote by $f_2$, as in Figure \ref{Fig2} (a). Necessarily the length of $f_2$ is bounded by $d$ and hence by $a$.  So $f_1 *\xi_1*\overline{f_2}$ is a loop based at $\eta(0)$ of length at most
    $$\ell(\eta\mid_{[0,(2k-1)a+\delta]})+a=(2k-1)a+\delta+a= 2ka+\delta.$$
    Because $V$ satisfies the effective loop shortening property with parameters $c$, $a$, and $k$, we can homotope $f_1*\xi_1 *\overline{f_2}$ to a loop $\gamma$ of length at most $(2k-2)a$ over loops based at $p$ of length at most $c$. Then by Lemma \ref{lemma:loops_to_arcs} there exists a path homotopy $\gamma_\tau$, $\tau\in[0,1]$, between $f_1*\xi_1$ and $\overline{\gamma} *f_2$ over curves of length at most $$\ell(\gamma_\tau)\leq\ell(f_2)+c\leq a+c \leq 2c,$$ see Figure \ref{Fig2} (b). This allows us to construct a family of short segments connecting the points of $f_1$ with $p$ as follows. For $t\leq (2k-1)a$, let $\beta_t=\eta\mid_{[0,t]}$. Thus $\beta_{(2k-1)a}=f_1$. Now, for $t\in[{(2k-1)a,(2k-1)a+\delta]}$ we set $t'=\frac{t-(2k-1)a}{(2k-1)a+\delta}\in[0,1]$ and define $\beta_t$ by following the curve $\gamma_{t'}$ and then the portion of $\overline{\xi_1}$ of length $\delta t'$ starting at $\xi_1(\delta)$. Note that $\beta_{(2k-1)a+\delta}=\overline{\gamma}*f_2$.
    \par 
    Now consider the curve $\eta_0$ formed by replacing the segment $f_1*\xi_1$ of $\eta$ by $\overline{\gamma}*f_2$. This curve has length at most 
    \begin{align*}
        \ell(\eta)-\ell(f_1)-\ell(\xi_1)+\ell(\gamma)+\ell(f_2)
        &\leq 
        \ell(\eta)-(2k-1)a-\delta+(2k-2)a+a
        \\
        &= 
        \ell(\eta)-\delta.
    \end{align*}
    We now parameterize $\eta_0$ by arc length and repeat the previous step, that is, if $\eta_0$ is longer than $(2k-1)a$, then we pick $\delta_0$ small enough to satisfy Definition \ref{def:shortening} and such that $\left(2k-1\right)a+\delta_0\leq \ell(\eta_0)$. We then connect $\eta(0)$ and $\eta_0((2k-1)a + \delta_0)$ by a minimizing geodesic $f_3$ and define $\xi_2$ as the arc of $\eta$ from $\eta((2k-1)a)$ to $\eta((2k-1)a+\delta_0)$, as in Figure \ref{Fig2} (c). We then repeat our shortening procedure. After a finite number of steps, we obtain the required family of curves. 
\end{proof}

\noindent Note that $\beta_1$, which joins $p$ to itself, is not necessarily a point curve. This is because as $\beta_t$ moves around the curve $\eta$, it may get wrapped around a topological or geometric obstruction.


\begin{figure}[ht]
  \includegraphics[width=\textwidth]{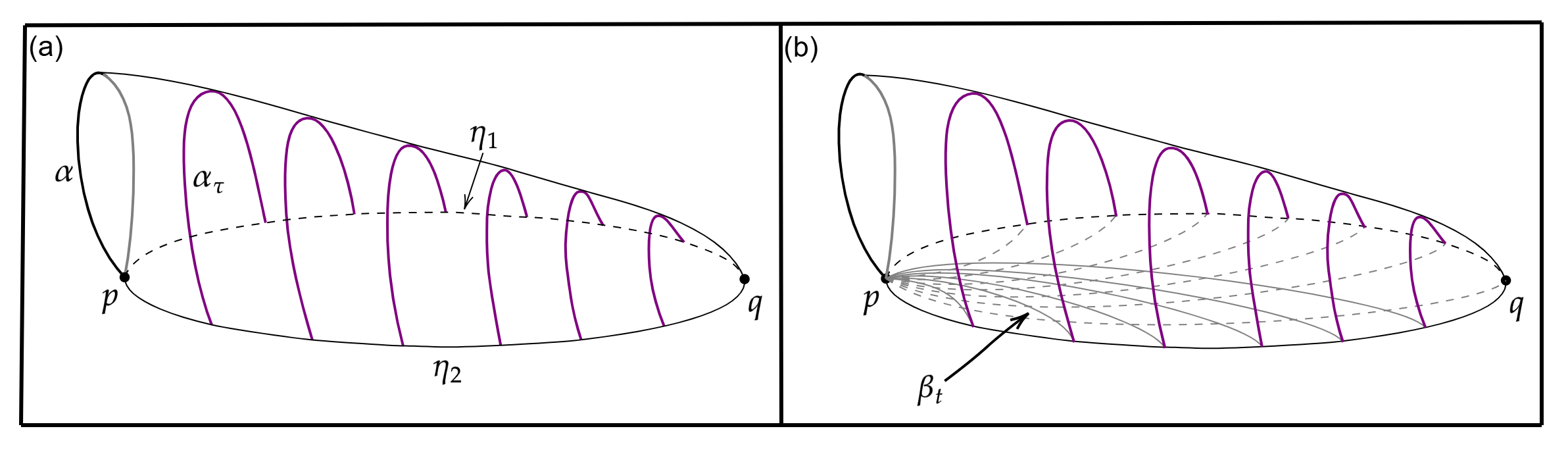}
    \caption{
    Illustration of the families of curves in the \hyperref[CSL]{Curve Shortening Lemma}. \textbf{(a)} The one-parameter family of curves $\alpha_{\tau}$ in $M$ with endpoints in $N$.
    \textbf{(b)} The family of curves $\beta_t$ in $N$ that connect the points of $\beta$ to the point $p$.}
  \label{fig:free_boundary_flow}
\end{figure}

\begin{CSL}
\label{CSL} 
    Let $M$ be a complete Riemannian manifold and $N$ a closed submanifold of $M$, with intrinsic diameter $D$.
    Suppose that there exist real numbers $c_0$, $a_0 \geq D$ and $k_0\geq 1$ such that for sufficiently small $\delta>0$, every loop based at a point in $N$ of length at most $2a_0k_0+\delta$ can be shortened in $N$ to a loop of length at most $(2k_0-2)a_0$ over loops with a fixed basepoint of length at most $c_0$.   
    \par
    Let $V \subset M$ be a compact set with extrinsic diameter $d$ that contains $N$. If $M$ has no stable geodesic chords orthogonal to $N$ of length at most $2k_0\max\{a_0,d\}$, then $V$ satisfies the effective loop shortening property restricted to $N$ with parameters $c=\max\{6c_0,2k_0d+5c_0\}$, $a=\max\{a_0,d\}$, and $k=k_0$.
\end{CSL}
    \begin{figure}[ht]
      \includegraphics[width=0.75\textwidth]{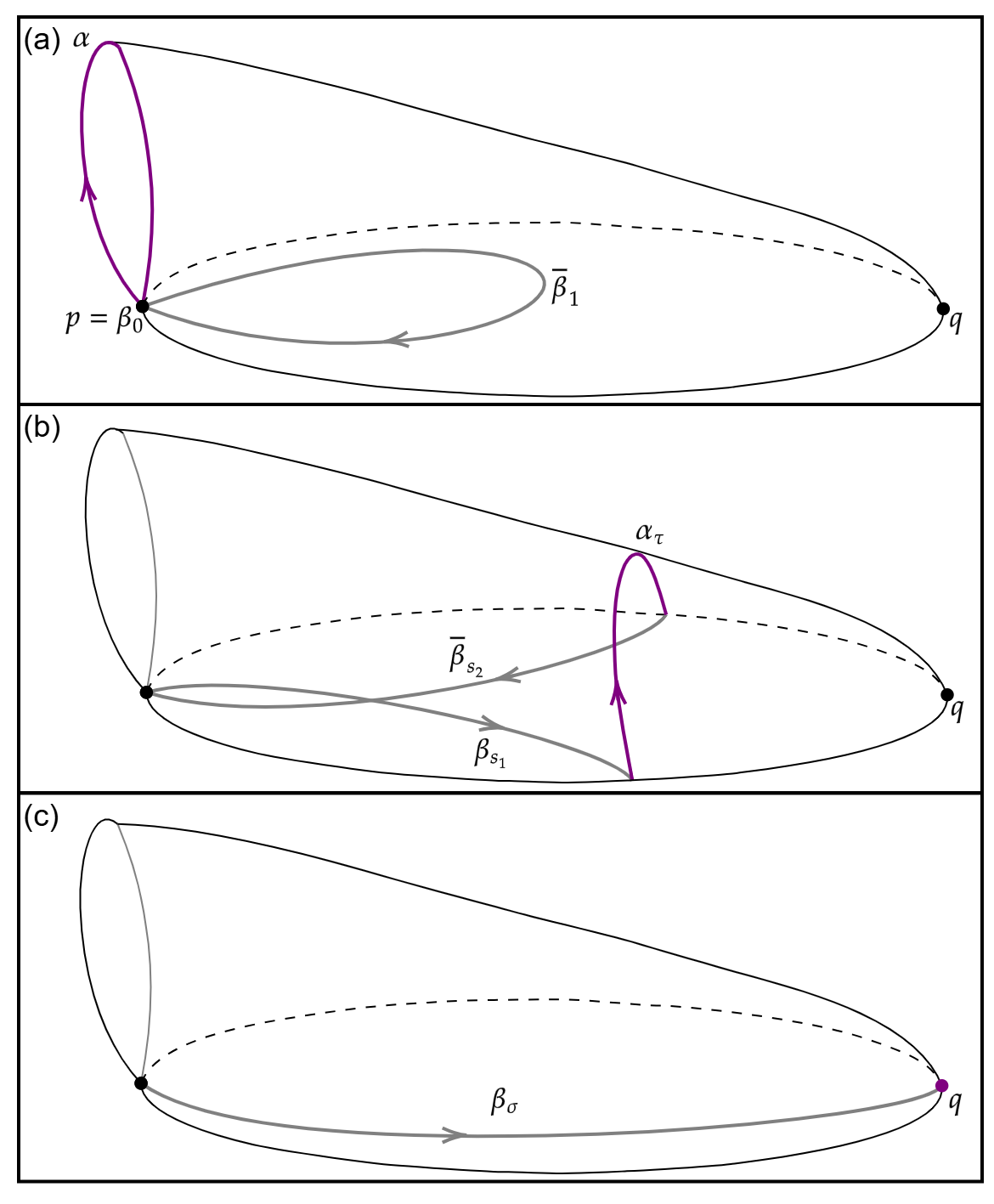}
        \caption{Construction of the homotopy from $\alpha$ to $p=\beta_0$ through loops based at $p$ in the proof of the \hyperref[CSL]{Curve Shortening Lemma}.}
      \label{fig:free_boundary_flow_2}
    \end{figure}
\begin{proof}
    Note that we must take $a=\max\{a_0,d\}$ to ensure that $a\geq d$, as required by the definition of the effective loop shortening property.
    Thus, consider a loop $\alpha$ in $V$ based on $N$ of length at most $2k_0\max\{a_0,d\}+\delta$. 
    First, we apply a free boundary curve shortening flow to $\alpha$. This is a curve shortening flow defined on curves with endpoints on the submanifold $N$. Under this flow, curves either contract in finite time or converge to a geodesic that is orthogonal to $N$ at its endpoints. For details of the free boundary curve shortening flow on manifolds with boundary, see \cite{ko2023} or Section 4 of \cite{hass1994}. This flow produces a one-parameter family of curves $\{\alpha_\tau\}_{\tau\in[0, 1]}$ in $M$ with endpoints on $N$. This family converges in finite time to some point $q$ (see Figure \ref{fig:free_boundary_flow} (a)), since by assumption there are no stable geodesic chords on $M$ orthogonal to $N$ of length less than or equal to $2k_0\max\{a_0,d\}$, and hence none of length less than or equal to $2k_0\max\{a_0,d\}+\delta$ for sufficiently small $\delta>0$.
    \par 
    Next, consider the trajectories $\eta_1(\tau)=\alpha_\tau(0)$ and $\eta_2(\tau)=\alpha_\tau(1)$. 
    Since the final curve in the family $\alpha_\tau$ is the point curve $q$, $\eta_1$ and $\eta_2$ share endpoints. Therefore we can define the closed curve $\eta=\eta_1*\overline{\eta_2}$ in $N$ parameterized by $[0,1]$. 
    By the effective loop shortening property, there exist real numbers $c_0$, $a_0 \geq D$ and $k_0\geq 1$ such that for sufficiently small $\delta>0$, every loop based at a point in $N$ of length at most $2a_0k_0+\delta$ can be shortened in $N$ to a loop of length at most $(2k_0-2)a_0$ over loops with a fixed basepoint of length at most $c_0$. Therefore
    we can apply Lemma \ref{lemma:base_case} to connect the points on $\eta$ with the point $p$ by a family of curves $\beta_t$ in $N$ for $t \in [0,1]$, as in Figure \ref{fig:free_boundary_flow} (b), and the length of the curves $\beta_t$ is bounded above by $2c_0$. Note that while $\beta_0$ is the constant curve $p$, $\beta_1$ is a possibly non-trivial loop in $N$ based at $p$, that, by Lemma \ref{lemma:base_case}, has length at most $(2k_0-2)a_0$, as in Figure \ref{fig:free_boundary_flow_2} (a).
    \par 
    Our final step is to construct a homotopy between $\alpha$ and $\beta_1$ through loops based at $p$ of length at most $6c_0$. Since $\beta_1$ is a loop of length at most $(2k_0-2)a_0$, this proves our claim. First note that $\alpha$ is homotopic to $\alpha*\overline{\beta_1} *\beta_1 $ through curves of length at most 
    $$
    \ell(\alpha)+2\ell(\beta_1)\leq 2k_0\max\{a_0,d\}+\delta +2(2k_0-2)a_0
    .$$
    Next, choose $\sigma$ so that $\beta_\sigma(1)=q$ (see Figure \ref{fig:free_boundary_flow_2} (c)). Then 
    $\alpha*\overline{\beta_1}=\beta_0*\alpha*\overline{\beta_1}$ 
    can be homotoped to $\beta_{\sigma}*\overline{\beta_{\sigma}}$ over loops of the form $\beta_{s_1}*\alpha_\tau*\overline{\beta_{s_2}}$, where the $s_i$ are defined by $\eta(s_1)=\alpha_\tau(0)$ and $\eta(s_2)=\alpha_\tau(1)$. Recall that the final curve in the family $\alpha_\tau$ is the point curve $q$ (see Figure \ref{fig:free_boundary_flow_2} (b)). Since $\ell(\alpha_\tau)\leq \ell(\alpha)$, these curves have length at most 
    $$\ell(\alpha_\tau)+2\max_t \ell(\beta_t)\leq 2k_0\max\{a_0,d\}+\delta + 4c_0.$$ 
    We can then contract $\beta_{\sigma}*\overline{\beta_{\sigma}}$ along itself to $p$, resulting in a homotopy from $\alpha*\overline{\beta_1}$ to $p$ through loops based at $p$.
    Thus,
    $\alpha*\overline{\beta_1} *\beta_1 $ is path homotopic to $\beta_1$ through loops based at $p$ of length at most 
    \begin{align*}
    \ell(\beta_1)+2k_0\max\{a_0,d\}+\delta + 4c_0
    &\leq 
    (2k_0-2)a_0+
    2k_0\max\{a_0,d\}+\delta + 4c_0\\
    &\leq 
    2k_0\max\{a_0,d\}+ 5c_0,
    \end{align*}
    since we have $2a_0k_0+\delta \leq c_0$ by the definition of the effective loop shortening property. 
    This completes our homotopy between $\alpha$ and $\beta_1$. Recall that the loop $\beta_1$ has length at most $(2k_0-2)a_0$. Therefore we have constructed a homotopy from the loop $\alpha$, which was an arbitrary loop of length at most $2k_0\max\{a_0,d\}+\delta$, to a loop of length at most $(2k_0-2)a_0$ through based point loops of length at most $2k_0\max\{a_0,d\}+5c_0$.     
    This proves the effective loop shortening property with the claimed parameters.
    Note that because $2a_0k_0\leq c_0$, we also have
    $$2k_0\max\{a_0,d\}+5c_0
    \leq 
    \max\{6c_0,2k_0d+5c_0\}
    \leq
    6c_0+2k_0d.\qedhere
    $$  
\end{proof} 

Let $N$ be as in the \hyperref[CSL]{Curve Shortening Lemma}. A special case of 
the above lemma occurs if every loop in $N$ of length at most $2D$ is contractible to a point curve over loops of length at most $c_0$ for some $c_0>0$. This means that $N$ satisfies the effective loop shortening property with parameters $c_0$, $a_0=D$, and $k_0=1$. In this instance, any compact subset $V\subset M$ of extrinsic diameter $d$ that contains $N$ satisfies the effective loop shortening property restricted to $N$ with parameters $c=\max\{6c_0,2d+5c_0\}$, $a=\max\{D,d\}$, and $k=1$. Thus, we immediately obtain the following corollary of the \hyperref[CSL]{Curve Shortening Lemma}.
\begin{cor} \label{cor:k=1}
    Let $M$ be a complete manifold and $N \subset M$ be a closed submanifold of intrinsic diameter $D$. Let $V$ be a compact subset of $M$ of extrinsic diameter $d$ containing $N$. Suppose that every loop in $N$ of length at most $2D$ is contractible over loops of length at most $c$. Suppose that there are no stable geodesic chords in $M$ orthogonal to $N$ of length at most $2\max\{D,d\}$. 
    Then $V$ satisfies the effective loop shortening property restricted to $N$ with parameters $\max\{6c,2d+5c\}, \max\{D,d\},$ and $1$. \qed
\end{cor}

We now explain how to utilize the effective loop shortening property to ensure the existence of short orthogonal geodesic chords. 
For convenience, we use the following definition due to Y. Liokumovich, D. Maximo, and the fourth author in \cite{lmr}.

 \begin{definition} 
    \label{def:delta_thin}  
    Consider a map $f: [-\epsilon, \epsilon]^m\times[0,1] \to M$. Let $\sigma_{x}$ be the linear path in $[-\epsilon, \epsilon]^m$ connecting $x \in \partial ([-\epsilon, \epsilon]^m)$ with the origin ${\bf0}=(0, \ldots, 0)$. Suppose there exists $\delta>0$ such that for each $(x,t)\in \partial ([-\epsilon, \epsilon]^m)\times [0,1]$ the length of $f(\sigma_{x},t)$ is at most $\delta$. Then $f$ is called a $\delta$-thin strip.  
 \end{definition}
 \noindent 
 Recall that $i\geq 1$ is the smallest number such that $\pi_{i}(\Omega_NM,N) \neq \{0\}$. We proceed by induction on $i$. 
 Our $i$-dimensional homotopy class is expressed as a map $\tilde{f}$ from the $(i+1)$-cell into $M$ with the image of the boundary of the cell lying in $N$. Given a $\delta>0$ we can subdivide the cell into $\delta$-thin strips. We then shorten the image of $\tilde{f}$ restricted to each individual strip while ensuring the strips agree on the boundaries so that they can be re-glued into a map homotopic to $\tilde{f}$. This is done in Lemma \ref{lemma:homotopic_to_short_map}, with an important technical result needed for the induction step proven in Lemma \ref{lemma:extending_delta_thin}. Since the case where $i=1$ can be proven both directly and by induction, we now state this case and give the direct proof. Note that for convenience we will instead deal with the map $f:D^i\times[0,1]\to M$ induced by including the parameter of each path in the image of $\tilde{f}$ into the domain of the map.
 
    \begin{figure}[ht]
        \centering
        \includegraphics[width=\textwidth]{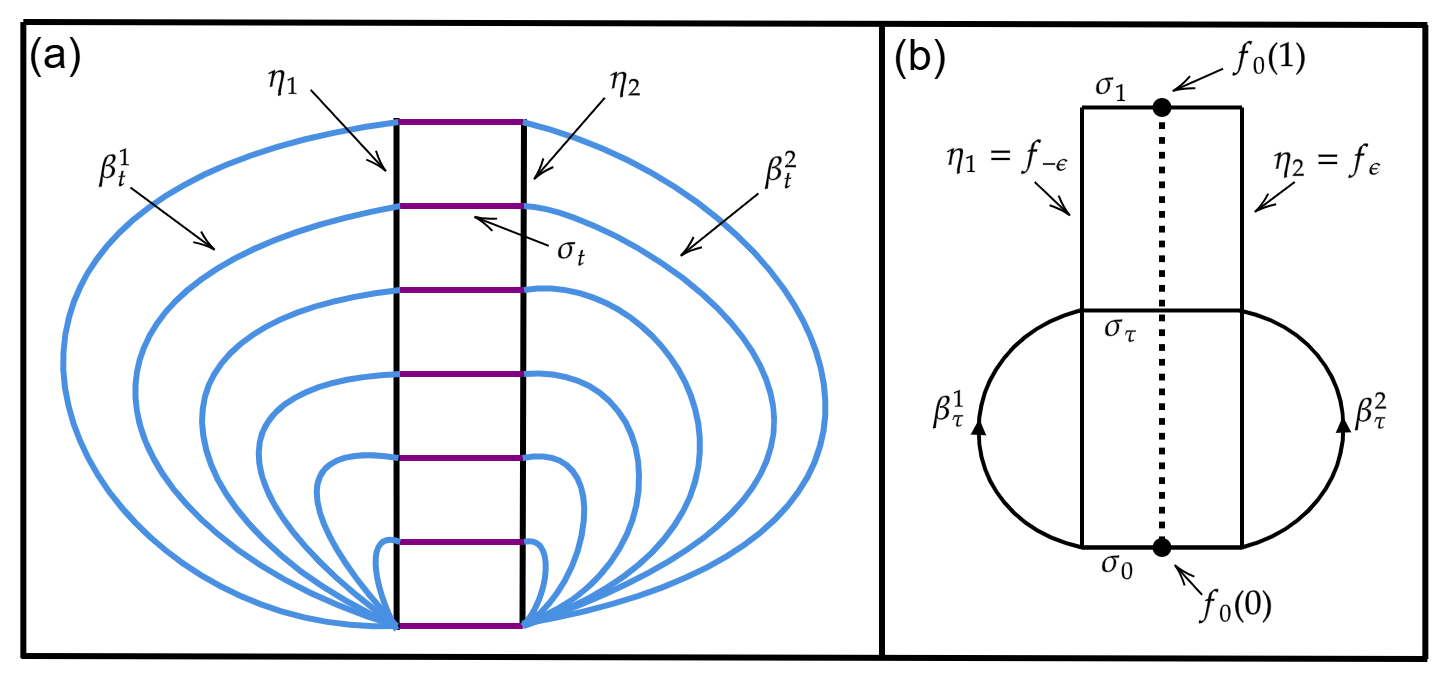}
        \caption{
        \textbf{(a)}
        The one-parameter family of curves $\beta_t^i$ of length at most $2c$.
        \textbf{(b)}
        The loop that needs to be filled to complete the homotopy $F_\tau$. 
        } 
        \label{fig:a_t_i}
    \end{figure}
 
\begin{lem} 
\label{lemma:example_case}
    Let $N$ be a closed submanifold of a complete manifold $M$.
    Let $f: [-\epsilon, \epsilon]\times[0,1] \to M$ be a $\delta$-thin strip such that  $f(x,0),f(x,1)\in N$ for every $x \in [-\epsilon, \epsilon]$. Suppose that 
    $V\subseteq M$ contains the image of $f$ and satisfies the effective loop shortening property restricted to $N$ with parameters $c$, $a$, and $k$. Then there exists a map $g: [-\epsilon, \epsilon]\times[0,1]\to M$ homotopic to $f$ relative to $f([-\epsilon, \epsilon]\times\{0,1\})\subset N$ such that the length of every curve in the image of $g$ is bounded by $6c+5\delta$. 
\end{lem}
    \begin{figure}[ht]
        \includegraphics[width=\textwidth]{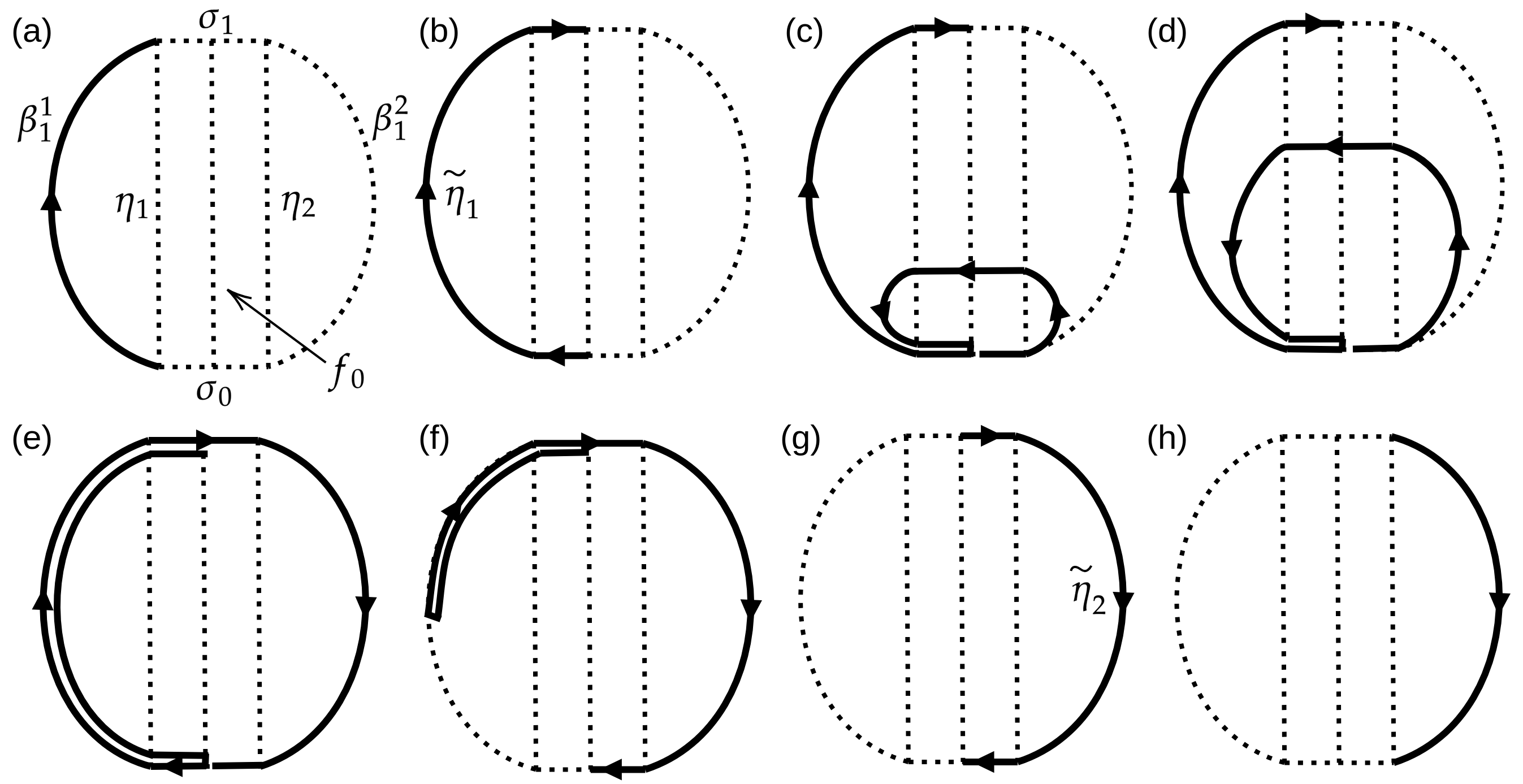}
        \caption{The construction of the homotopy between $\beta_1^1$ and $\beta_1^2$.}
        \label{e_i_homotopy}
    \end{figure}

\begin{proof}
    For convenience, we denote $f(x,t)$ by $f_x(t)$. Let $\eta_1(t)=f_{-\epsilon}(t)$ and $\eta_2(t)=f_{\epsilon}(t)$. Because $M$ satisfies the effective loop shortening property restricted to $N$ with parameters $c$, $a$, and $k$, for $i=1,2$, by Lemma \ref{lemma:base_case} there exists a one-parameter family of curves $\beta^i_t$ such that $\beta^i_t(0)=p$, $\beta^i_{t}(1)=\eta_i(t)$, and the length of each $\beta^i_t$ is at most $2c$ (see Figure \ref{fig:a_t_i} (a)).  
    Define the curve $s\mapsto \sigma_t(s)$ as $s\mapsto f_s(t)$. Note that because $f$ is a $\delta$-thin strip, each $\sigma_t$ has length at most $\delta$. We will show that
    \begin{align*}
        \tilde{\eta}_1=\overline{\sigma_0}|_{[-\epsilon,0]}*\beta^1_1*\sigma_1|_{[-\epsilon,0]}
    \end{align*}
    and
    \begin{align*}
        \tilde{\eta}_2=\sigma_0|_{[0,\epsilon]}*\beta^2_1*\overline{\sigma_1}|_{[0,\epsilon]}
    \end{align*}
    are homotopic relative to $f([-\epsilon,\epsilon]\times\{0,1\})$ (see Figure \ref{e_i_homotopy}). This homotopy is through curves that have length at most $6c+5\delta$ and whose endpoints lie on $N$. We define $g$ to be this homotopy. 
    \par 
    The homotopy proceeds as follows. The curve $\beta^1_1$ is homotopic to $\tilde{\eta}_1$ relative to $f([-\epsilon,\epsilon]\times\{0,1\})$ by moving both of the endpoints $\beta^1_1(0)$ and $\beta^1_1(1)$ along $\sigma_0(t)$ and $\sigma_1(t)$ respectively, ending at $\sigma_0(0)$ and $\sigma_1(0)$ as in Figure \ref{e_i_homotopy} (b). Similarly, $\beta^2_1$ is relatively homotopic to $\tilde{\eta}_2$. Thus, in order to complete our homotopy we need to ``fill in'' the domain formed by $\tilde{\eta}_1$ and $\tilde{\eta}_2$ with interpolating curves. 
    This domain is bounded by the loop based at $f_{0}(0)$ given by 
    \begin{align*}
        \tilde{\eta}_1*\overline{\tilde{\eta}_2}=\overline{\sigma_0}|_{[-\epsilon,0]}*\beta^1_1 *\sigma_1*\overline{\beta^2_1}*\overline{\sigma_0}|_{[0,\epsilon]}.
    \end{align*}
    We can contract this loop to $\overline{\sigma_0}|_{[-\epsilon,0]}*\sigma_0*\overline{\sigma_0}|_{[0,\epsilon]}\subset N$ over the family of loops
    \begin{align*}
        \overline{\sigma_0}|_{[-\epsilon,0]} *\beta^1_{1-t}*\sigma_{1-t}*\overline{\beta^2_{1-t}}*\overline{\sigma_0}|_{[0,\epsilon]}
    \end{align*}
    for $t\in[0,1]$.
    Note that the lengths of these loops do not exceed $4c+3\delta$.
    Thus by Lemma \ref{lemma:loops_to_arcs}, $\tilde{\eta}_1$ and $\tilde{\eta}_2$ are homotopic relative to $f([-\epsilon,\epsilon]\times\{0\})$ through curves of length at most 
    \begin{align*}
        \ell(\tilde{\eta}_1)+4c+3\delta\leq 6c+5\delta,
    \end{align*} 
     as in Figure \ref{e_i_homotopy} (b)--(g). We reparameterize each individual curve by the unit interval in this homotopy and reparameterize the homotopy itself by $[-\epsilon,\epsilon]$. The resulting homotopy defines $g$.
    \par    
    We now prove that there exists a homotopy $F_\tau(x,t):([-\epsilon, \epsilon]\times[0,1])\times[0,1]\to M$ from $F_0=f$ to $F_1=g$. We do this by using the two families of curves $\{\beta^i_t\}_{t\in[0,1]}, i\in\{1, 2\}$ but change the variable to $\tau$, thus creating two ``new" families of curves $\{\beta^i_{\tau}\}_{\tau\in[0,1]}, i\in\{1, 2\}$.
    We then build the homotopy for each $\tau\in [0, 1]$ by shortening the original curves and slowly replacing them by curves that interpolate between the curves $\beta^1_\tau$ and $\beta^2_\tau$ as follows. 
    Consider the loop
    \begin{align*}
        \overline{\sigma_0}|_{[-\epsilon,0]}*\beta^1_\tau *\sigma_\tau*\overline{\beta^2_\tau}*\overline{\sigma_0}|_{[0,\epsilon]},
    \end{align*}
    which is shown in Figure \ref{fig:a_t_i} (b). For each $\tau$, we can construct a continuous family of curves $s\mapsto\xi_s^\tau$ between $\beta^1_\tau$
    and $\beta^2_\tau$ with endpoints on $f(\{0,\tau\}\times[0,1])$,
    as we did above in the case that $\tau=1$ (see Figure \ref{fig:eta_tau}). In fact, we can take the portion of the homotopy obtained above starting at $\tau$ instead of 1. Let $r_{x,\tau}$ be the arc of $\sigma_{x,\tau}$ connecting $\xi_x^\tau(1)$ and $f_x(\tau)$.
    We then define 
    $$F_\tau(x,\cdot)=\xi_x^\tau
    *r_{x,\tau}*f_x(s)|_{s\in[\tau, 1]}*\sigma_1|_{[x,\epsilon]},$$ 
    with parameterization by the unit interval (see Figure \ref{fig:F_tau}). Note that indeed $F_0=f$ and $F_1=\xi_1=g$ mod $N$. 
\end{proof}

    \begin{figure}[ht]
        \includegraphics[width=0.6\textwidth]{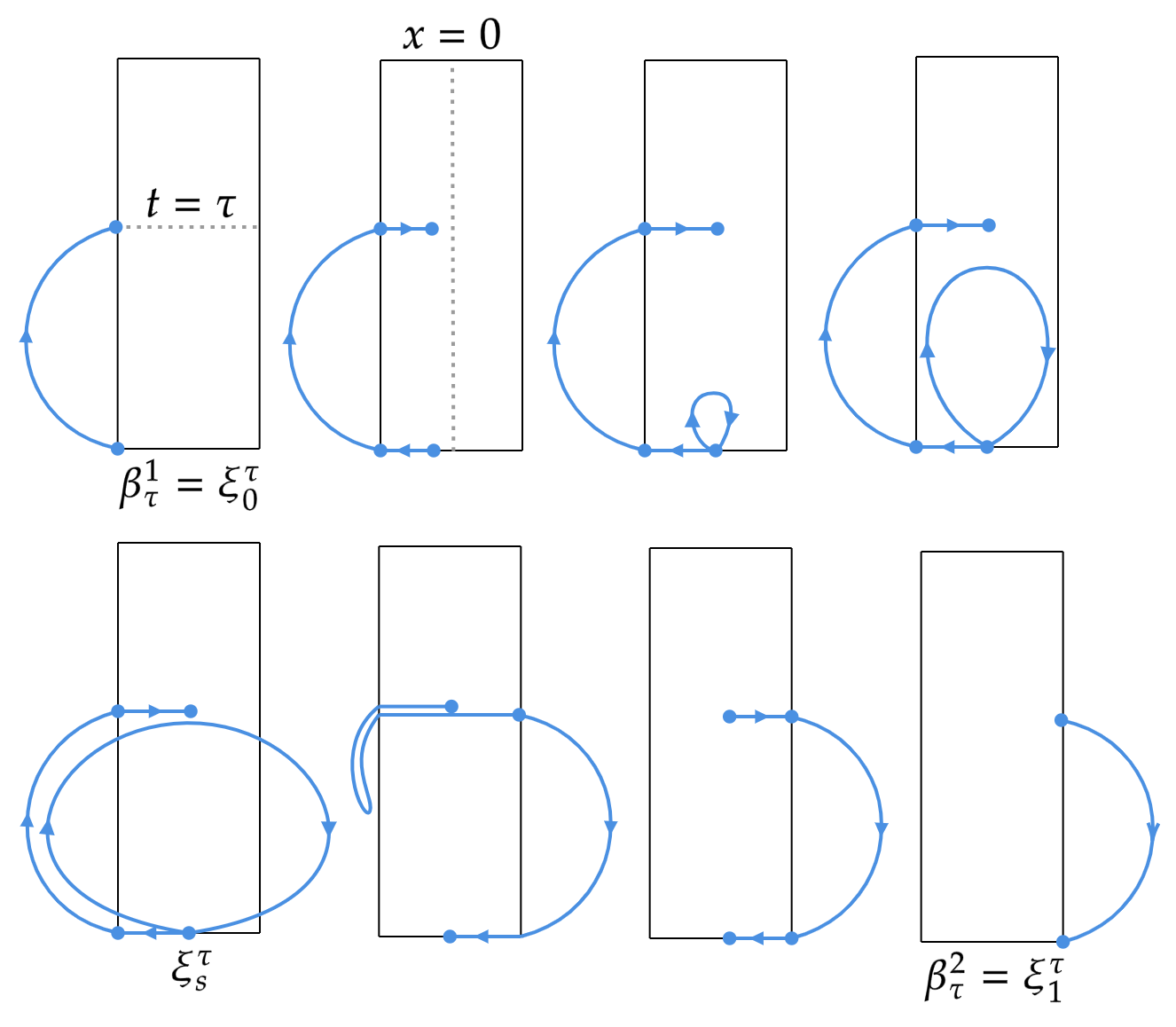}
        \caption{
        An example of the homotopy $s\mapsto \xi^\tau_s$ between $\beta^1_\tau=\xi^\tau_0$ and $\beta^2_\tau=\xi^\tau_1$.} 
        \label{fig:eta_tau}
    \end{figure}
    
    \begin{figure}[ht]
        \includegraphics[width=0.75\textwidth]{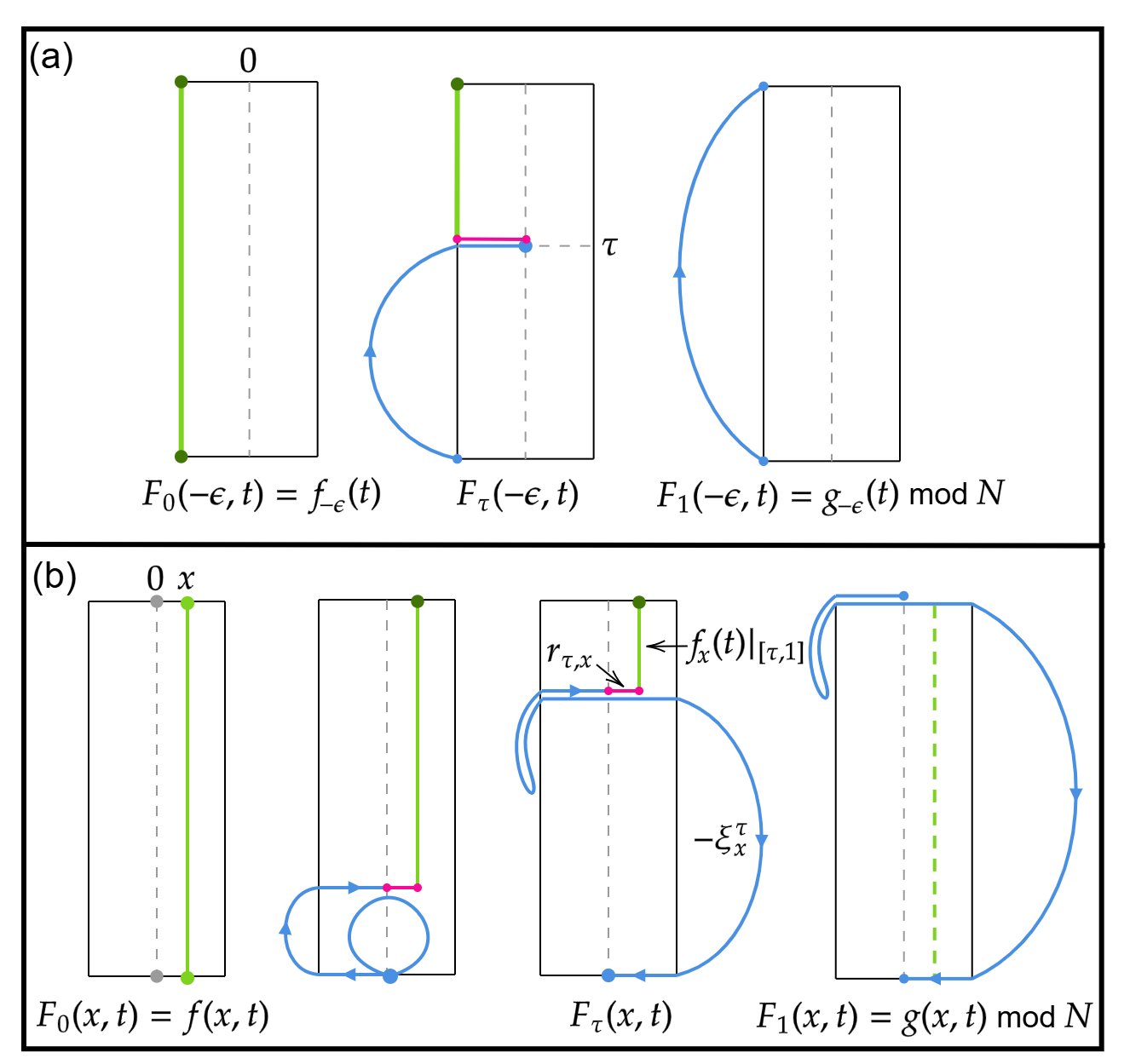}
        \caption{
        Two examples of the homotopy $\tau\mapsto F_{\tau}(x,t)$. \textbf{(a)} The homotopy $\tau\mapsto F_{\tau}(-\epsilon,t)$. \textbf{(b)} The homotopy $\tau\mapsto F_{\tau}(x,t)$ for some $0<x<\epsilon$.} 
        \label{fig:F_tau}
    \end{figure}

    The following technical lemma will help us define the homotopy between $\tilde{f}$ and the shorter map $\tilde{f}$ we will construct in Lemma \ref{lemma:homotopic_to_short_map}.
    Note that this lemma is a minor variation of Lemma 5.3 in \cite{linear_bounds_group_paper}.

\begin{lem} 
    \label{lemma:extending_delta_thin}
    Let $N$ be a closed submanifold of a complete manifold $M$. Suppose we have the following two maps:
    \begin{itemize}
        \item 
        Let 
        $$f:  [-\epsilon, \epsilon]^m\times[0,1] \to M$$ 
        be a $\delta$-thin strip such that for each $x \in [-\epsilon, \epsilon]^m$, $f_x(0) \in N$. 
        \item
        Let $$H:\partial ([-\epsilon, \epsilon]^m) \times [0,1] \to\Omega^{L}M$$ such that $H(x,0)\in N$ and $H(x,1)=f_x(1)$.
    \end{itemize}
    Suppose that $V\subseteq M$ contains the image of $f$ and satisfies the effective loop shortening property with parameters $c$, $a$, and $k$. Then for some small $\delta>0$, one can extend $H$ to $$H':[-\epsilon, \epsilon]^m \times [0,1]\to \Omega^{L^1} M,$$ such that $ H'(x,0)\in N$, $H'(x,1)=f_x(1)$, and $L^1= L+4c+2\delta$.
\end{lem}

\begin{proof}
    We extend $H$ to $[-\epsilon,\epsilon]^m \times [0,1]$ as follows. First, note that by Lemma \ref{lemma:base_case}, for $\textbf{0} =(0,\ldots,0)\in [-\epsilon, \epsilon]^m$ there is a one-parameter family of curves $\beta_t$ such that $\beta_t(0)=f_{\textbf{0}}(0)$, $\beta_t(1)=f_{\textbf{0}}(t)$, and each $\beta_t$ has length bounded by $2c$. Given $x\in[-\epsilon, \epsilon]^m$, we show how to define $H'(x,t)$. Let $\sigma_x$ be the line segment connecting $x$ and $\textbf{0}$. Let $s(x) = 
    2(|x| - \epsilon/4)/\epsilon$ and define $\sigma_{x,t}=f(\sigma_x,t)$. 
    Note that the path $\sigma_{x,t}$ is a higher-dimensional generalization of the path $\sigma_t$ used in the previous lemma.
    Lastly, let $y$ be the boundary point on the ray from $\textbf{0}$ through $x$. 
    Then we define
    \begin{align*}
        H'(x,t)&=
        \beta_{t}
        *\sigma_{x, t} 
    \end{align*}
    when $|x|\in[0,\epsilon/4]$,
    \begin{align*}
        H'(x,t)=
        H(y, s(x)t)
        *\overline{\sigma_{y, s(x)t}}
        *\overline{\beta_{s(x)t}}
        *\beta_{t}
        *\sigma_{x, t} 
    \end{align*}
    when $|x|\in(\epsilon/4,3\epsilon/4]$, and
    \begin{align*}
        H'(x,l)&=
        H(y, t)
        *\overline{\sigma_{y, t}}
        *\overline{\beta_{t}}
        *\beta_{t}
        *\sigma_{x, t}
    \end{align*}
    when $|x|\in(3\epsilon/4,\epsilon]$. Note that $H'(x,0)\in N$ and $H'(x,1)=f_x(1)$ as desired. Moreover, $H'(x,t)$ has length at most
    \begin{align*}
        L+2\max_t\ell(\beta_t)+2\max_{x,t}\ell(\sigma_{x,t})
        \leq L+4c+2\delta,
    \end{align*}
    as claimed.
\end{proof}

In the next lemma, we establish the base case and the induction step the main lemma required for the proof of the main result. 

\begin{lem} 
\label{lemma:homotopic_to_short_map} 
    Let $N$ be a closed submanifold of a complete Riemannian manifold $M$. 
   Consider a map $\tilde{f}:(D^i, \partial D^i) \to (\Omega_NM,N)$ and its induced map $f:D^i\times[0,1] \to M$.
    Let $V$ be a compact subset of $M$ of extrinsic diameter $d$ that contains $N$ and the image of $f$.
    Suppose $V$ satisfies the effective loop shortening property restricted to $N$ with parameters $c, a$, and $k$. Then for any $\delta>0$ there exists a map $\tilde{g}:(D^i, \partial D^i) \to (\Omega_N^LM, N)$, homotopic to $\tilde{f}$, such that $L=2c(2i+1)+\delta$.
\end{lem}
\begin{proof}
    Fix some $\delta>0$ and partition $D^i$ into very small rectangles $R_i$ so that the restriction of $\tilde{f}$ to each $R_i$ is a $\delta$-thin strip in $M$. 
    We construct $\tilde{g}$ by induction on the skeleta of $D^i$. 
    In a similar fashion, we construct by induction a homotopy $F_t$ such that $F_0=\tilde{f}$ and $F_1=\tilde{g}$.
    
    We begin with the base case of our induction, which is the definition of $\tilde{g}$ on the $0$-skeleton of $R_i$. By Lemma \ref{lemma:base_case}, given a vertex $v_j$ in the $0$-skeleton of $R_i$ and the curve $\eta_j=\tilde{f}(v_j)$, there exists a family of curves $\beta^j_t$ in $M$ of length at most $2c$ continuously connecting $\eta_j(0)$ with the points $\eta_j(t)$. We define $\tilde{g}(v_j)=\beta^j_1$. Note that because $\beta^j_1$ and $\tilde{f}(v_j)$ share endpoints, we indeed have $\beta^j_1\in \Omega_NM$. We need to check that this makes $\tilde{g}$ into a map of pairs. Consider the boundary of $D^i$, which is mapped to $N$ by $\tilde{f}$. Any point in the boundary is mapped to a point curve by $\tilde{f}$. By the proof of Lemma \ref{lemma:base_case}, each curve in the family $\beta_t$ obtained by applying Lemma \ref{lemma:base_case} to a point curve is also a point curve. This ensures that the boundary of $D^i$ is mapped to $N$ by $\tilde{g}$. This concludes the base case.
   \par
    Now suppose we have extended $\tilde{g}$ to the $(k-1)$-skeleton of $D^i$ for $k\geq 1$. We extend it to the $k$-skeleton of $D^i$ as follows. The map $\tilde{g}$ restricted to the boundary of a $k$-cell $D^k\simeq [-\epsilon,\epsilon]^k$ induces a map $g:\partial([-\epsilon,\epsilon]^k)\times[0,1] \to M$. By applying Lemma \ref{lemma:extending_delta_thin} to $g$ and the $\delta$-strip given by $f|_{R_i}$, we obtain a map $H^{(k)}:D^k \times [0,1] \to \Omega^L M$ with
    \begin{align*}
        L\leq 2c + k(4c+2\delta)
    \end{align*}
    that extends $g$ and such that each path $H^{(k)}(x,t)$ begins at $f_x(0)$ and ends at $f_x(t)$. We then define $\tilde{g}(x)$ at a point $x$ on the $k$-skeleton by $H^{(k)}(x,1)$. Moreover, the curves in the image of $g$ on the $k$-skeleton have length at most $2c + k(4c+2\delta)$. 
    \par
    We now need to show that $\tilde{g}$ is homotopic to $\tilde{f}$. In order to do that we define a homotopy $F_\tau:D^i\to \Omega M$ as follows. At the final step of the above induction process, we defined the map $H^{(i)}$. On each $R_i$ we define $$F_\tau(x)=H^{(i)}(x,\tau)*f_x(s)|_{s\in[\tau, 1]}.$$
    Thus $F_0(x)=\tilde{f}(x)$ and $F_1(x)=H^{(i)}(x,1)=\tilde{g}(x)$, as claimed. Redefining $\delta$ gives the claimed length bound.
\end{proof}

We are now ready to prove Theorem \ref{MainTheorem1}.
For easy reference, we restate it here.

\begin{thm}
\label{theorem:proof_of_thm_1}
    Let $M$ be a complete $n$-dimensional manifold and $N \subsetneq M$ a closed submanifold such that $\pi_i(\Omega_NM,N)\neq 0$ for some $i>0$. 
    Consider a homotopically non-trivial map $\tilde{f}:(D^i, \partial D^i) \to (\Omega_NM,N)$ and its induced map $f:D^i\times[0,1] \to M$.
    Let $V$ be a compact subset of $M$ of extrinsic diameter $d$ that contains $N$ and the image of $f$.
    Suppose $V$ has the effective loop shortening property restricted to $N$ with parameters $c, a,$ and $k$. 
  Then there exists a geodesic chord that is orthogonal to $N$ of length at most $2c(2i+1)$. 
    \label{thm:ogcs_exist}
\end{thm}
\begin{proof}
We first consider the case where $\pi_1(M,N) \neq \{0\}$. The existence of a geodesic chord orthogonal to $N$ is straightforward. We can simply consider an element of $\Omega_NM$, the space of piecewise differentiable curves in $M$ with endpoints on $N$, representing a non-trivial class in $\pi_1(M,N)$. By applying a length shortening flow in $\Omega_NM$, see, for example, Section 3 in \cite{zhou2016}, this curve converges either to an orthogonal geodesic chord or to a point. Thus, if we can estimate the length of this curve, we then automatically have a bound for the length of a shortest orthogonal geodesic chord on $(M,N)$.

We now assume that $\pi_1(M, N)$ is trivial. By assumption, there exists a non-trivial homotopy class of $\pi_i(\Omega_NM,N)$ for some $i\geq2$. The manifolds $M$ and $N$ satisfy the hypotheses of Lemma \ref{lemma:homotopic_to_short_map}, so this class can be represented by a map $g$ whose image consists of curves of length at most $2c(2i+1)+\delta$ for any $\delta>0$. Therefore there is a representative whose image consists of curves of length at most $2c(2i+1)$. When we apply the free boundary curve shortening flow to the curves in the image of $g$, some curve must fail to contract or else we would be able to homotope $g$ to a constant map. The only way such a curve can fail to contract is if it converges to a geodesic chord on $M$ that is orthogonal to $N$, that  necessarily is of length at most $2c(2i+1)$.
\end{proof}

\section{Applications of Theorem \ref{MainTheorem1}}
\label{sec:applications}

In this section we prove Theorem \ref{thm:2sphere_in_ndisk} and we include the following series of explicit applications of Theorem \ref{MainTheorem1}, that we believe are of general interest. We start with the proof of Theorem \ref{thm:2sphere_in_ndisk}.

\begin{theorem_12}
    Let $M$ be a closed Riemannian manifold of dimension $n$ and diameter $d$. Let $N$ be a $2$-dimensional sphere of area $A(N)$ and intrinsic diameter $D$ embedded in $M$. Then there exists a geodesic chord on $M$ orthogonal to $N$ of length at most 
    $$
    (4d+96D +8232\sqrt{A(N)})(2n+1).
    $$
\end{theorem_12}
\begin{proof}
    By the work of Y. Liokumovich, A. Nabutovsky and the fourth author in \cite{lio_2015}, any simple loop $\gamma$ bounding a $2$-disk $\Delta$ can be contracted to a point through loops with fixed base point of length at most 
    \begin{align*}
        2\ell(\gamma)+686\sqrt{A(\Delta)}+2D_\Delta.
    \end{align*}
    Note that given a 2-disk $\Delta\subset N$, $\Delta$ has intrinsic diameter $D_\Delta  \leq D+\ell(\partial\Delta)/2$. This is because two points in $\Delta$ are connected by a geodesic in $N$ of length at most $D$, from which we can obtain a curve in $\Delta$ by replacing any segments outside of $\Delta$ by appropriately chosen segments of $\partial\Delta$. Consequently, any simple loop on $N$ of length at most $2D$ bounding $\Delta$, can be contracted through loops with fixed base point of length at most 
    \begin{align*}
        2\ell(\gamma)+686\sqrt{A(\Delta)}+2D_\Delta 
        &\leq
        2(2D)+686\sqrt{A(\Delta)}+2(D+\ell(\gamma)/2)\\
        &\leq  8D+686\sqrt{A(N)}.
    \end{align*}
    Therefore $N$ satisfies the effective loop shortening property with parameters $c=8D +686\sqrt{A(N)}$, $D$ and $1$. Recall that a stable geodesic is a stable critical point of the length functional. By Corollary \ref{cor:k=1}, there are two cases, Case 1, where $M$ admits a stable geodesic chord orthogonal to $N$ of length at most 
    $2\max\{d,D\}$ and Case 2, where $M$ satisfies the effective loop shortening property with parameters $\max\{6c,5c+2d\}$, $\max\{d,D\}$ and $1$. In Case 1, the result automatically holds.
    Suppose then that we are in Case 2. Since $M$ and $N$ are closed, there is some $i\leq n$ such that $\pi_i(M,N)$ is non-trivial. Therefore by Theorem \ref{MainTheorem1}, $M$ admits a geodesic chord orthogonal to $N$ of length at most 
    $$
    2\max\{6c,2d+5c\}(2n+1)\leq (4d+12c)(2n+1)=
    (4d+96D +8232\sqrt{A(N)})(2n+1).
    $$
    Thus the claimed bound holds.
\end{proof}
The remaining applications can be viewed as corollaries of Theorem \ref{MainTheorem1}. We include the proofs for the convenience of the reader. In particular, all of them treat the case where either the manifold or the submanifold have bounded geometry. 
\begin{cor}
    \label{cor:submanifold_3sphere}  
    Let $N$ be a $3$-sphere with $|\sec (N)| \leq 1$, $\vol(N)\geq v$, and $\diam(N)\leq D$. Then the following hold:  
    \begin{enumerate}
        \item  
        For any closed submanifold $Q\subset N$, there is a geodesic chord on $N$ orthogonal to $Q$ of length at most $28D\phi^2,$
        where $\phi$ is a computable constant.  
        \item  
          If $N$ is a submanifold of a compact $n$-dimensional manifold $M$ with $n\geq 4$ and diameter $d$, then there is an orthogonal geodesic chord on $M$ orthogonal to $N$ of length at most
        \begin{align*}
        4\max\{d,D\}(6\phi^2+1)(2n+1),
        \end{align*}
        where $\phi$ is a computable constant.  
    \end{enumerate}
\end{cor}

\begin{proof}
    By Proposition 1 of \cite{rotman_3sphere} by Nabutovsky and the fourth author, there is a bi-Lipschitz homeomorphism from $N$ to the round $S^3$ such that both the homeomorphism and its inverse have Lipschitz constant at most $\phi=\exp(\exp( 10^5 n_0^2))$, where $n_0=c\exp(16\diam N)/v^{6}$. In the round $S^3$, the only closed geodesics are the great circles, that are not minima of the length functional. Therefore the length shortening flow in $S^3$ contracts every curve in $S^3$ to a point. 
    \par 
    Now, consider a loop of length at most $2L$ in $N$ for any fixed $L>0$. We can map this curve to a curve in the round $S^3$ using our Lipschitz map. This new curve has length at most $2L\phi$, and by our above remark can be homotoped to a point in $S^3$ through curves of length at most $2L\phi$. We then map the curves in this homotopy back to $N$ via our Lipschitz map. This proves that $N$ satisfies the effective loop shortening property with parameters $2L\phi^2$, $L$ and $1$. 
    \par
    We next consider the first claim. If $Q$ is a closed submanifold of $N$, then by Proposition \ref{prop:pi_non_zero}, $\pi_i(N, Q)$ is non-trivial for some $i\leq \dim(N)=3$. Setting $L=D$, we see that $N$ satisfies the effective loop shortening property with parameters $2D\phi^2$, $D$ and $1$. Then by Theorem \ref{MainTheorem1}, there is a geodesic chord on $N$ orthogonal to $Q$ of length at most $14(2D\phi^2)$.
    \par
    Consider the second statement. Setting $L=\max\{d,D\}$, $N$ satisfies the effective loop shortening property with parameters $2\max\{d,D\}\phi^2$, $\max\{d,D\}$ and $1$. Without loss of generality, assume that $M$ does not contain any stable geodesic chords orthogonal to $N$ of length at most $2 \max\{d,D\}$. Then by Corollary \ref{cor:k=1}, $M$ satisfies the effective loop shortening property restricted to $N$ with parameters $12\max\{d,D\}\phi^2+2d$, $\max\{d,D\}$ and $1$.
    Thus by Theorem \ref{MainTheorem1}, there is an orthogonal geodesic chord in $M$ orthogonal to $N$ with length at most $4\max\{d,D\}(6\phi^2+1)(2n+1)$.
\end{proof}

\begin{cor}
    \label{cor:submanifold_4dim}  
    Let $N$ be a $4$-dimensional, closed, simply connected Riemannian manifold 
    with $|\Ric (N)| \leq 3$, $\vol(N)\geq v$, and $\diam(N)\leq D$. 
    Then the following hold:    
    \begin{enumerate}
        \item  
        For any closed submanifold $Q\subset N$, there is a geodesic chord on $N$ orthogonal to $Q$ of length at most
        \begin{align*}
            18(2D+5h(v,D)) 
        \end{align*}
        \item   
        If $N$ is a submanifold of a compact $n$-dimensional manifold $M$ with $n\geq 5$ and diameter $d$, then there is an orthogonal geodesic chord on $M$ orthogonal to $N$ of length at most
        \begin{align*}
            (24D+60h(v,D)+4d)(2n+1).
        \end{align*}  
    \end{enumerate}
\end{cor}

\begin{proof}
    By the work of N. Wu and Z. Zhu in \cite{wu_zhu_2022}, there exists a function $h(v,D)$ such that any loop $\gamma$ on 
    $N$ can be contracted to a point via a homotopy of width at most $h(v,D)$ (see Theorem 1.11 A in \cite{wu_zhu_2022}). By this we mean that under the contraction homotopy $f_t$, the trajectory $t\mapsto f_t(x)$ of any point $x$ is a curve of length at most $h(v,D)$. Let $a$ be a constant 
    such that $2a\leq h(v,D)$. Then by Lemmas 4.1 and 4.2 of \cite{linear_bounds_group_paper}, any loop $\gamma$ on $N$ of length at most $2a$ can be contracted to a point through fixed point loops with length at most $2a+5h(v,D)+\epsilon$. 
    \par
    We now prove Part 1. If $Q$ is a closed submanifold of $N$, then, as in the proof of Corollary \ref{cor:submanifold_3sphere}, we have $\pi_i(N, Q)$ non-trivial for some $i\leq \dim(N)\leq 4$. Moreover $N$ satisfies the effective loop shortening property (restricted to $Q$) with parameters $2D+5h(v,D)+\epsilon, D$ and 1. Thus by Theorem \ref{MainTheorem1}, there is a geodesic chord on $N$ orthogonal to $Q$ of length at most $18(2D+5h(v,D)+\epsilon)$. Since this is true for all $\epsilon>0$, there is in fact a geodesic chord on $N$ orthogonal to $Q$ of length at most $18(2D+5h(v,D))$, as claimed.
    \par
    We now prove Part 2. By Corollary \ref{cor:k=1}, either $M$ admits a geodesic chord orthogonal to $N$ with length at most $2\max\{d,D\}$, or $M$ satisfies the effective loop shortening property restricted to $N$ with parameters $6(2D+5h(v,D)+\epsilon)+2d$, $\max\{d,D\}$ and $1$.
    Theorem \ref{MainTheorem1} then proves that $M$ admits a geodesic chord orthogonal to $N$ with length at most 
    $$\left(12(2D+5h(v,D)+\epsilon\right)+4d)(2n+1)$$
    Since this holds for all $\epsilon>0$, there must also be a geodesic chord orthogonal to $N$ of length at most
    $$(24D+60h(v,D)+4d)(2n+1).\qedhere$$
\end{proof}

\begin{cor} 
    \label{cor:ricci}
    Let $N$ be a closed $k$-dimensional Riemannian manifold with $$\Ric \geq \frac{k-1}{r^2}$$ for some $r>0$. Then the following hold.
    \begin{enumerate}
        \item  
        For any closed submanifold $Q\subset N$, there is a geodesic chord on $N$ orthogonal to $Q$ of length at most
        \begin{align*}
            8\pi r(2k+1)
        \end{align*}
          \item   
          If $N$ is a submanifold of a closed $n$-dimensional manifold $M$ with $n\geq k+1$ and diameter $d$, there is an orthogonal geodesic chord on $M$ orthogonal to $N$ of length at most
        \begin{align*}
            (8d+48\pi r)(2n+1).
        \end{align*}  
    \end{enumerate} 
    
\end{cor}
\begin{proof}
    By the Bonnet-Myers Theorem, the diameter of $N$ is bounded above by $\pi r$. By the same token, any curve in $N$ of length greater than $\pi r$ can be path homotoped to a curve of length at most $\pi r$. In particular, any curve of length at most $4\pi r$ can be path homotoped to a curve of length at most $\pi r$ through curves of length at most $4\pi r$. Therefore $N$ satisfies the effective loop shortening property with parameters $4\pi r$, $\pi r$, and $2$. 
    \par 
    We begin with the proof of Part 1. If $Q$ is a closed submanifold of $N$, then we have $\pi_i(N, Q)$ non-trivial for some $i\leq \dim(N)= k$. Then by Theorem \ref{MainTheorem1}, there is a geodesic chord on $N$ orthogonal to $Q$ of length at most $8\pi r(2k+1)$. This proves our first claim.
    \par 
    We now prove Part 2. 
    Because $M$ is closed, $\pi_i(M,N)\not=0$ by Proposition \ref{prop:pi_non_zero}. By the \hyperref[CSL]{Curve Shortening Lemma}, either $M$ admits a geodesic chord orthogonal to $N$ of length $2\max\{d,\pi r\}$, or $M$ satisfies the effective loop shortening property restricted to $N$ with parameters $\max\{24\pi r, 4d+20\pi r\}$, $\max\{\pi r, d\}$, and $2$. Note that 
    $$\max\{24\pi r, 4d+20\pi r\}\leq 24\pi r + 4d.$$
    Therefore by Theorem \ref{MainTheorem1} there is an orthogonal geodesic chord on $M$ orthogonal to $N$ of length at most $2(24\pi r + 4d)(2n+1)$. This proves the second claim.
\end{proof}

Theorem \ref{MainTheorem1} also applies to manifolds $M$ with boundary as follows. Any compact manifold $M$ with convex boundary admits an orthogonal geodesic chord on $(M,\partial M)$ by \cite{gluck1984}. If $\partial M$ has more than one connected component, then we can use the free boundary curve shortening flow to obtain an orthogonal geodesic chord with bounded length. This is a curve shortening flow defined on curves with endpoints on the boundary of a given manifold. Under this flow, curves either contract in finite time or converge to an orthogonal geodesic chord. Therefore an application of the free boundary flow to a minimizing geodesic connecting two different components of $\partial M$ produces an orthogonal geodesic chord of length at most $\diam M$. If $M$ is not simply connected, we can take a shortest homotopically non-trivial curve with endpoints on $\partial M$. In the case of a simply connected manifold with one boundary component, we apply the techniques of Theorem \ref{MainTheorem1}. Because orthogonal geodesic chords on manifolds with boundary are not allowed to internally intersect the boundary, our techniques do not guarantee the existence of such a curve unless $\partial M$ is convex. In general, we are only able to produce a geodesic segment starting and ending at $\partial M$ that otherwise lies in the interior of $M$, which we call a geodesic chord, and that moreover is orthogonal to $\partial M$ at at least one endpoint. However, we are still able to bound its length.

\begin{cor}    
\label{cor:boundary}
    Let $M$ be a compact Riemannian manifold of dimension $n$ with connected boundary.
    Suppose that $M$ satisfies the effective loop shortening property restricted to $\partial M$ with parameters $c$, $a$ and $k$. Then there exists a geodesic chord on $M$ that is orthogonal to $\partial M$ at (at least) one end and has length at most $2c(2n+1)$. If $\partial M$ is convex, this curve is in fact an orthogonal geodesic chord.
\end{cor} 
\begin{proof}
    Because $M$ is compact with a unique boundary component, by Proposition \ref{prop:pi_non_zero} $\pi_i(M, \partial M)$ is non-trivial for some $i>0$. If $\partial M$ is convex, then our result follows immediately from Theorem \ref{MainTheorem1}, as convexity prevents a geodesic from touching the boundary tangentially. 
    If $\partial M$ is not convex, we can apply the following argument due to Seifert \cite{seifert_1949} (see also Theorem B of \cite{gluck1984} and Theorem 2 of \cite{beach2024_2d_case}). We can attach a collar to $\partial M$ that is parameterized by geodesics of length $\epsilon$ that emanate orthogonally from $\partial M$. We call the new manifold $M_\epsilon$ and view $M$ as a subset of $M_\epsilon$. Moreover, we can modify the metric on the collar to make $M_\epsilon$ a manifold with convex boundary. This metric also has the property that geodesics orthogonal to $\partial M_\epsilon$ cross $\partial M\subset M_\epsilon$ orthogonally. Consequently, any geodesic that starts and ends orthogonally at $\partial M_\epsilon$ must cross $\partial M$ orthogonally at at least two points.
    \par 
    By Lemma \ref{lemma:homotopic_to_short_map}, there is a homotopically non-trivial map $g:(D^i,\partial D^i)\to(\Omega_{\partial M} M,\partial M)$ whose image passes through curves of length at most $2c(2i+1)+\delta$. We can extend this to a homotopically non-trivial map $g_\epsilon:(D^i,\partial D^i)\to(\Omega_{\partial M_\epsilon} M_\epsilon,\partial M_\epsilon)$. To do so, we connect each endpoint of a curve in the image of $g$ with the nearest point on $\partial M_\epsilon$ via a minimizing geodesic. Necessarily, these geodesics have length at most $\epsilon$. Moreover, these geodesics vary continuously, ensuring that $g_\epsilon$ continuous. Thus, the image of $g_\epsilon$ consists of curves of length at most $2c(2i+1)+\delta + 2\epsilon$. By the proof of Theorem \ref{MainTheorem1}, the map $g_\epsilon$ gives rise to an orthogonal geodesic chord $\gamma_\epsilon$ on $M_\epsilon$ orthogonal to $\partial M_\epsilon$ of length at most $2c(2i+1)+\delta + 2\epsilon$. By the previous paragraph, the initial arc of $\gamma_{\epsilon}\cap M$ is non-empty and starts orthogonally to $\partial M$. Moreover, it is a geodesic of length at most $2c(2i+1)+\delta + 2\epsilon$. Taking the limit as $\delta$ and $\epsilon$ tend to zero, we obtain our result.
\end{proof}

We also have an analog of Theorem \ref{thm:2sphere_in_ndisk}.

\begin{cor} 
    \label{cor:boundary_3disk}
    Let $M$ be a Riemannian $3$-disk with convex boundary and diameter $d$. Suppose $\partial M$ has volume $A$ and intrinsic diameter $D$. Then there exists an orthogonal geodesic chord on $M$ of length at most
    $$
    28(d+24D +2058\sqrt{A}).
    $$
\end{cor}
\begin{proof}
    The boundary of $M$ is a 2-sphere, so by the proof of Theorem \ref{thm:2sphere_in_ndisk} we know that either $M$ admits a stable geodesic chord orthogonal to $\partial M$ of length at most  $2\max\{d,D\}$
    or $M$ satisfies the effective loop shortening property with parameters $\max\{6c,2d+5c\}$, 
    $\max\{d,D\}$
    and $1$, where $c=8D +686\sqrt{A}$. 
    Therefore by Corollary \ref{cor:boundary} there exists a curve satisfying the claimed properties of length at most 
    $$2(2d+6c)(2n+1)=28(d+24D +2058\sqrt{A}).\qedhere$$
\end{proof}

\noindent 
\textbf{Acknowledgments.} 
The authors gratefully acknowledge support from the Princeton IAS Summer Collaborators program in 2024, where the work on this paper was begun. This work was also supported by the National Science Foundation under Grant No. DMS-1928930, while the I. Beach, R. Rotman and C. Searle were in residence at the Simons Laufer Mathematical Sciences Institute (formerly MSRI) in Berkeley, California, during the Fall 2024 semester. A part of this paper was written while I. Beach and R. Rotman were at the Metric Geometry trimester program at the Hausdorff Research Institute for Mathematics. I. Beach was supported in part by an NSERC Canada Graduate Scholarships Doctoral grant. H. Contreras-Peruyero was supported by UNAM Postdoctoral Program (POSDOC) and SECIHTI Postdoctoral Fellowship 2025. R. Rotman was supported by NSERC Discovery Grant and Simons Foundation International Award SFI-MPS-SFM-00006548, C. Searle was partially supported by NSF Grant DMS-2506633 and DMS-2204324.

\bibliographystyle{abbrvnat}
\bibliography{biblio}

\end{document}